\documentclass[11pt,reqno]{amsart}

\usepackage{amsmath,amssymb,enumerate,graphicx,hyperref}
\allowdisplaybreaks

  \scrollmode

\newcommand{\dif}{\mathrm d}
\newcommand{\dprime}{\prime\prime}

\newcommand{\R}{{\mathbb R}}
\newcommand{\N}{{\mathbb N}}

\newcommand{\EE}{{\mathbb E}}
\newcommand{\PP}{{\mathbb P}}

\newcommand{\ind}{1}
\newcommand{\usn}{\underline {s}_n}
\newcommand{\utn}{\underline {t}_n}
\newcommand{\urn}{\underline {r}_n}
\newcommand{\sgn}{\operatorname{sgn}}
\newcommand{\eps}{\varepsilon}

\newcommand{\F}{{\mathcal F}}

\theoremstyle{plain}
\newtheorem{theorem}{Theorem}[section]
\newtheorem{prop}[theorem]{Proposition}
\newtheorem{lemma}[theorem]{Lemma}
\newtheorem{assumption}[theorem]{Assumption}

\begin{document}
\title[]{Strong convergence of the tamed Euler scheme for scalar SDEs with superlinearly growing and discontinuous drift coefficient}

\author[Hu]
{Huimin Hu}
\address{
School of Mathematics and Statistics \\
Central South University \\
Changsha \\
China} \email{huhuimin@csu.edu.cn}

\author[Gan]
{Siqing Gan}
\address{
School of Mathematics and Statistics \\
Central South University \\
Changsha \\
China} \email{sqgan@csu.edu.cn}

\begin{abstract}
In this paper, we consider scalar stochastic differential equations (SDEs) with a superlinearly growing and piecewise continuous drift coefficient. Existence and uniqueness of strong solutions of such SDEs are obtained. Furthermore, the classical $L_p$-error rate $1/2$ for all $p \in [1,\infty)$ is recovered for the tamed Euler scheme. A numerical example is provided to support our conclusion.


\end{abstract}
\maketitle

\section{Introduction}

Consider the autonomous stochastic differential equation (SDE)
\begin{equation}\label{sde0}
\begin{aligned}
\dif X_t &= \mu(X_t) \, \dif t + \sigma(X_t) \, \dif W_t, \quad t \geq 0, \\
X_0 &= x_0
\end{aligned}
\end{equation}
with initial value $x_0$, drift coefficient $\mu : \R\to\R$, diffusion coefficient $\sigma : \R\to\R$ and 1-dimensional driving Brownian motion $W$.

It is well-known that if the coefficients $\mu$ and $\sigma$ satisfy the global Lipschitz condition then the SDE \eqref{sde0} admits a unique strong solution $X$ and the Euler--Maruyama scheme achieves an $L_p$-error rate $1/2$, for all $p \in [1,\infty)$, at any given time $T>0$. For brevity, we consider $T=1$ henceforth.

If the coefficients $\mu$ and $\sigma$ satisfy the local Lipschitz condition and the Khasminskii-type condition, then the SDE \eqref{sde0} admits a unique strong solution $X$, see e.g. \cite[Theorem 2.3.6]{XM}. Unfortunately, for a large class of SDEs with superlinearly growing coefficients, the Euler--Maruyama scheme converges neither in the strong mean square sense nor in the numerically weak sense to the exact solution at a finite time point, see \cite[Theorem 1]{MH2}. In addition, the implementation of the implicit Euler method requires more computational effort. As a consequence, some explicit methods based on modifications of the usual Euler and Milstein schemes are proposed, see \cite{QG3,QG2,QG,LH,MH,XM2,XM3,SS2,SS,MT,XW,ZZ,XZ}.

If the drift coefficient $\mu$ is piecewise Lipschitz continuous and the diffusion coefficient $\sigma$ is Lipschitz continuous and non-zero at the potential discontinuity points of $\mu$, then the SDE \eqref{sde0} admits a unique strong solution $X$, see \cite[Theorem 2.2]{GL}. The $L_p$-error of the Euler--Maruyama scheme for such SDEs has been studied in recent years, see \cite{GL2,TM,AN}. In particular, in \cite{TM} the classical $L_p$-error rate $1/2$ for all $p \in [1,\infty)$ is recovered for the Euler--Maruyama scheme.

Existence, uniqueness and approximation of the strong solution of the SDE \eqref{sde0} in the case of superlinearly growing coefficients $\mu$ and $\sigma$ in the presence of discontinuities of the drift coefficient $\mu$ has been studied in \cite{TM2}. More precisely, it is assumed that $\mu$ is piecewise locally Lipschitz continuous and $\sigma$ is locally Lipschitz continuous and non-zero at the discontinuity points of $\mu$. Moreover, $\mu$ and $\sigma$ satisfy a piecewise monotone-type condition and a coercivity condition and the Lipschitz constants of both $\mu$ and $\sigma$ satisfy a polynomial growth condition. Under these conditions, it is proved that the SDE \eqref{sde0} admits a unique strong solution $X$ and the modified tamed Euler scheme proposed in \cite{SS} achieves an $L_p$-error rate $1/2$ for a suitable range of values of $p$.

In this article we prove the existence and uniqueness result for scalar SDEs under much weaker conditions. More precisely, we assume that the drift coefficient $\mu$ is piecewise locally Lipschitz continuous and the diffusion coefficient $\sigma$ is locally Lipschitz continuous and non-zero at the discontinuity points of $\mu$. Moreover, $\mu$ and $\sigma$ satisfy the Khasminskii-type condition.

In order to study the $L_p$-error of the tamed Euler scheme proposed in \cite{MH}, we impose stronger conditions on the coefficients $\mu$ and $\sigma$. More precisely, we assume that $\mu$ is piecewise locally Lipschitz continuous and $\sigma$ is globally Lipschitz continuous and non-zero at the discontinuity points of $\mu$. Moreover, $\mu$ satisfies a piecewise monotone-type condition and the Lipschitz constant of $\mu$ satisfies a polynomial growth condition. Under these conditions, the classical $L_p$-error rate $1/2$ for all $p \in [1,\infty)$ is recovered.

The rest of this paper is organized as follows. The existence and uniqueness result is established in the next section. The error estimates are presented in Section 3. A numerical example is provided in Section 4.

\section{Existence and uniqueness}

Throughout the whole article we assume that the following setting is fulfilled.
Let $(\Omega,\F,\PP)$ be a probability space with a normal filtration $(\F_t)_{t \geq 0}$, and let $W : [0,\infty) \times \Omega \to \R$ be an $(\F_t)_{t \geq 0}$-Brownian motion. Moreover, we suppose that $x_0 : \Omega \to \R$ is an $\F_{0}$-measurable random variable and $\mu,\sigma : \R \to \R$ are measurable functions. Here and below we use $\N_0$ to denote the set of all nonnegative integers, $\N$ to denote the set of all positive integers and $\|\cdot\|_{\infty}$ to denote the $L_{\infty}$-norm on the space of real-valued, continuous functions on $[0,1]$. For two real numbers $a$ and $b$, we use $a\wedge b = \min(a,b)$ and $a\vee b = \max(a,b)$.

We consider the SDE
\begin{equation}\label{sde1}
\begin{aligned}
\dif X_t &= \mu(X_t) \, \dif t + \sigma(X_t) \, \dif W_t, \quad t \geq 0, \\
X_0 &= x_0.
\end{aligned}
\end{equation}

\begin{assumption}\label{assum1}
We assume that the initial value $x_0$ satisfies
\[
\EE\big[|x_0|^2\big] < \infty,
\]
and the coefficients $\mu$ and $\sigma$ satisfy the following three conditions.
\begin{enumerate}
\item[\emph{(A1)}] There exist $k\in\N_0$ and $\xi_0,\dots,\xi_{k+1} \in [-\infty,\infty]$ with $-\infty = \xi_0 < \xi_1 < \cdots < \xi_k < \xi_{k+1} = \infty$ such that $\mu$ is locally Lipschitz continuous on the interval $(\xi_{i-1},\xi_i)$ for all $i \in \{ 1,\dots,k+1 \}$,
\item[\emph{(A2)}] $\sigma$ is locally Lipschitz continuous on $\R$ and $\sigma(\xi_i) \neq 0$ for all $i \in \{ 1,\dots,k \}$,
\item[\emph{(A3)}] there exists $c \in (0,\infty)$ such that for all $i \in \{ 1,\dots,k+1 \}$ and all $x \in (\xi_{i-1},\xi_i)$,
\[
x \cdot \mu(x) + \frac12 |\sigma(x)|^2 \leq c \cdot \big(1+|x|^2\big).
\]
\end{enumerate}
\end{assumption}

We introduce a transformation $G: \R\to\R$ that is used to switch from the SDE \eqref{sde1} to an SDE with continuous coefficients.

\begin{lemma}\label{transform}
Let Assumption \ref{assum1} hold. Then there exists a function $G : \R \to \R$ with the following properties:
\begin{enumerate}[\em (i)]
\item $G$ is differentiable with
\[
0 < \inf_{x\in\R} G^{\prime}(x) \leq \sup_{x\in\R} G^{\prime}(x) < \infty.
\]
In particular, $G$ is Lipschitz continuous and has an inverse $G^{-1}: \R\to\R$ that is Lipschitz continuous as well.
\item The derivative $G^{\prime}$ of $G$ is Lipschitz continuous hence absolutely continuous. Moreover, $G^{\prime}$ has a bounded Lebesgue-density $G^{\dprime}: \R\to\R$ that is Lipschitz continuous on each of the intervals $(\xi_0,\xi_1),\dots,(\xi_k,\xi_{k+1})$ and such that the functions
\[
\widetilde{\mu} = \left( G^{\prime} \cdot \mu + \frac12 G^{\dprime} \cdot \sigma^2 \right) \circ G^{-1} \quad 
\textrm{and} \quad 
\widetilde{\sigma} = (G^{\prime} \cdot \sigma) \circ G^{-1}
\]
are continuous.
\item $G(x)=x$ for $|x|$ sufficiently large.
\end{enumerate}
\end{lemma}

\begin{proof}
See the proof of \cite[Lemma 7]{TM}.
\end{proof}

Next, choose $G$ according to Lemma \ref{transform} and define $\widetilde{\mu}$ and $\widetilde{\sigma}$ according to Lemma \ref{transform}(ii).

\begin{lemma}
Let Assumption \ref{assum1} hold. Then $\widetilde{\mu}$ and $\widetilde{\sigma}$ are locally Lipschitz continuous, and there exists $c \in (0,\infty)$ such that for all $x\in\R$,
\[
x \cdot \widetilde{\mu}(x) + \frac12 |\widetilde{\sigma}(x)|^2 
\leq c \cdot \big(1+|x|^2\big).
\]
\end{lemma}

\begin{proof}
It is straightforward to check that $\widetilde{\mu}$ and $\widetilde{\sigma}$ are locally Lipschitz continuous.

By Lemma \ref{transform} we obtain that there exists $n_0\in\N$ such that for all $x\in\R$ with $|x| > n_0$,
\[
\widetilde{\mu}(x) = \mu(x), \quad \widetilde{\sigma}(x) = \sigma(x).
\]
Hence, by Lemma \ref{assum1} we obtain that there exists $c_1 \in (0,\infty)$ such
that for all $x\in\R$ with $|x| > n_0$,
\begin{equation}\label{trcoe1}
x \cdot \widetilde{\mu}(x) + \frac12 |\widetilde{\sigma}(x)|^2 
\leq c_1 \cdot \big(1+|x|^2\big).
\end{equation}

Using the continuity of $\widetilde{\mu}$ and $\widetilde{\sigma}$, we see that there exists $c_2 \in (0,\infty)$ such that for all $x\in\R$ with $|x| \leq n_0$,
\[
x \cdot \widetilde{\mu}(x) + \frac12 |\widetilde{\sigma}(x)|^2 
\leq c_2,
\]
which, together with \eqref{trcoe1}, completes the proof.
\end{proof}

\begin{lemma}
Let Assumption \ref{assum1} hold. Then the SDE
\begin{equation}\label{sde2}
\begin{aligned}
\dif Z_t &= \widetilde{\mu}(Z_t) \, \dif t + \widetilde{\sigma}(Z_t) \, \dif W_t, \quad t \geq 0, \\
Z_0 &= G(x_0)
\end{aligned}
\end{equation}
has a unique strong solution $Z$.
\end{lemma}

\begin{proof}
See e.g. \cite[Theorem 2.3.6]{XM}.
\end{proof}

\begin{theorem}
Let Assumption \ref{assum1} hold. Then the SDE \eqref{sde1} has a unique strong solution
\[
X = G^{-1} \circ Z,
\]
where $Z$ is the unique strong solution of the SDE \eqref{sde2}.
\end{theorem}

\begin{proof}
By applying the It\^{o} formula, see e.g. \cite[Problem 3.7.3]{IK}, to $G^{-1}$ we conclude that there exists a unique strong solution to \eqref{sde1}.
\end{proof}

\section{Error estimate for the tamed Euler scheme}

In this section, we adapt the proof techniques in \cite{MH,TM} to provide an $L_{p}$-error estimate for the tamed Euler scheme. We define
\[
t_n = \lfloor n \cdot t \rfloor / n
\]
for every $n\in\N$ and every $t \in [0,1]$.

In order to analyze the $L_p$-error of the tamed Euler scheme for the SDE \eqref{sde1}, we introduce the following assumption on the initial value and the coefficients, which is stronger than Assumption \ref{assum1}.

\begin{assumption}\label{assum2}
Let $\ell \in (0,\infty)$. We assume that the initial value $x_0$ satisfies that
\begin{equation*}
\EE\big[|x_0|^p\big] < \infty, \quad \forall p \in [1,\infty),
\end{equation*}
and the coefficients $\mu$ and $\sigma$ satisfy the following three conditions.
\begin{enumerate}
\item[\emph{(B1)}] There exist $k\in\N_0$ and $\xi_0,\dots,\xi_{k+1} \in [-\infty,\infty]$ with $-\infty = \xi_0 < \xi_1 < \cdots < \xi_k < \xi_{k+1} = \infty$ such that $\mu$ is locally Lipschitz continuous on the interval $(\xi_{i-1},\xi_i)$ for all $i \in \{ 1,\dots,k+1 \}$,
\item[\emph{(B2)}] $\sigma$ is Lipschitz continuous on $\R$ and $\sigma(\xi_i) \neq 0$ for all $i \in \{ 1,\dots,k \}$,
\item[\emph{(B3)}] there exists $c \in (0,\infty)$ such that for all $i \in \{ 1,\dots,k+1 \}$ and all $x,y \in (\xi_{i-1},\xi_i)$,
\begin{align*}
&(x-y) \cdot (\mu(x)-\mu(y)) \leq c \cdot |x-y|^2, \\
&|\mu(x)-\mu(y)| \leq c \cdot \big(1+|x|^{\ell}+|y|^{\ell}\big) \cdot |x-y|.
\end{align*}
\end{enumerate}
\end{assumption}

For $n\in\N$ let $\widehat{X}_n = (\widehat{X}_{n,t})_{t \in [0,1]}$ denote the time-continuous tamed Euler scheme with step-size $1/n$ associated to the SDE \eqref{sde1}, i.e. $\widehat{X}_n$ is recursively given by $\widehat{X}_{n,0} = x_0$ and
\[
\widehat{X}_{n,t} = \widehat{X}_{n,i/n} + \frac{\mu(\widehat{X}_{n,i/n})}{1 + 1/n \cdot |\mu(\widehat{X}_{n,i/n})|} \cdot (t-i/n) + \sigma(\widehat{X}_{n,i/n}) \cdot (W_t-W_{i/n})
\]
for $t \in (i/n,(i+1)/n]$ and $i \in \{ 0,\dots,n-1 \}$.

\subsection{\texorpdfstring{$L_{p}$-estimates of the solution and the time-continuous tamed Euler scheme}{Lp-estimates of the solution and the time-continuous tamed Euler scheme}}

\begin{lemma}
Let Assumption \ref{assum2} hold and $p \in [1,\infty)$. Then there exists $c \in (0,\infty)$ such that for all $\delta \in [0,1]$ and all $t \in [0,1-\delta]$,
\[
\Big(\EE\Big[ \sup_{s \in [t,t+\delta]} |X_s-X_t|^p \Big]\Big)^{1/p} \leq c \cdot \sqrt{\delta}.
\]
In particular,
\[
\EE\big[ \|X\|_{\infty}^p \big] < \infty.
\]
\end{lemma}

\begin{proof}
Without loss of generality we may assume that $p \in [2,\infty)$. By Assumption \ref{assum2} and \cite[Theorem 2.4.1]{XM} we obtain that for all $q \in [1,\infty)$,
\begin{equation}\label{solprop}
\sup_{t \in [0,1]} \EE\big[ |X_t|^q \big] < \infty.
\end{equation}
Using the fact that for all $a,b\in\R$,
\begin{equation}\label{ine}
|a+b|^p \leq 2^{p-1} \big(|a|^p+|b|^p\big),
\end{equation}
we deduce then that
\begin{align*}
&\EE\Big[ \sup_{s \in [t,t+\delta]} |X_s-X_t|^p \Big] \\
&\leq 2^{p-1} \EE\bigg[ \sup_{s \in [t,t+\delta]} \bigg|\int_t^s \mu(X_r) \, \dif r\bigg|^p \bigg] 
+ 2^{p-1} \EE\bigg[ \sup_{s \in [t,t+\delta]} \bigg|\int_t^s \sigma(X_r) \, \dif W_r\bigg|^p \bigg].
\end{align*}
Hence, using the H\"{o}lder inequality, \cite[Theorem 1.7.2]{XM} and Assumption \ref{assum2} we conclude that there exists $c \in (0,\infty)$ such that
\begin{align*}
&\EE\Big[ \sup_{s \in [t,t+\delta]} |X_s-X_t|^p \Big] \\
&\leq (2\delta)^{p-1} \EE\bigg[ \int_t^{t+\delta} |\mu(X_r)|^p \, \dif r \bigg] 
+ \frac12 \bigg(\frac{2 p^3}{p-1}\bigg)^{p/2} \delta^{(p-2)/2} \EE\bigg[ \int_t^{t+\delta} |\sigma(X_r)|^p \, \dif r \bigg] \\
&\leq c \delta^p \cdot \Big(1 + \sup_{t \in [0,1]} \EE\big[ |X_t|^{(1+\ell)p} \big]\Big) 
+ c \delta^{p/2} \cdot \Big(1 + \sup_{t \in [0,1]} \EE\big[ |X_t|^p \big]\Big),
\end{align*}
which, together with \eqref{solprop}, completes the proof.
\end{proof}

For every $x\in\R$, let $X^x$ denote the unique strong solution of the SDE
\begin{equation}\label{sde3}
\begin{aligned}
\dif X_t^x &= \mu(X_t^x) \, \dif t + \sigma(X_t^x) \, \dif W_t, \quad t \geq 0, \\
X_0^x &= x,
\end{aligned}
\end{equation}
and for all $x\in\R$ and $n\in\N$ we use $\widehat{X}_n^x = (\widehat{X}_{n,t}^x)_{t \in [0,1]}$ to denote the time-continuous tamed Euler scheme with step-size $1/n$ associated to the SDE \eqref{sde3}, i.e. $\widehat{X}_{n,0}^x = x$ and
\[
\widehat{X}_{n,t}^x = \widehat{X}_{n,i/n}^x + \frac{\mu(\widehat{X}_{n,i/n}^x)}{1 + 1/n \cdot |\mu(\widehat{X}_{n,i/n}^x)|} \cdot (t-i/n) + \sigma(\widehat{X}_{n,i/n}^x) \cdot (W_t-W_{i/n})
\]
for $t \in (i/n,(i+1)/n]$ and $i \in \{ 0,\dots,n-1 \}$. Furthermore, the
integral representation
\begin{equation}\label{intrep}
\widehat{X}_{n,t}^x = x + \int_0^t \frac{\mu(\widehat{X}_{n,\usn}^x)}{1 + 1/n \cdot |\mu(\widehat{X}_{n,\usn}^x)|} \, \dif s + \int_0^t \sigma(\widehat{X}_{n,\usn}^x) \, \dif W_s
\end{equation}
holds for every $n\in\N$ and $t \in [0,1]$.

\begin{lemma}\label{lambda}
Let Assumption \ref{assum2} hold. Then there exists $\lambda \in [4,\infty)$ such that
\begin{enumerate}[\em (i)]
\item 
\[
1+|\mu(x)|+|\sigma(x)| \leq \lambda
\]
for all $x\in\R$ with $|x| \leq 1$,
\item 
\[
|\sigma(x)|^2 \leq \lambda|x|^2
\]
for all $x\in\R$ with $|x| \geq 1$,
\item 
\[
x \cdot \mu(x) \leq \sqrt{\lambda}|x|^2
\]
for all $x\in\R$ with $|x| \geq 1$,
\item 
\[
|\mu(x)|^2 \leq n\sqrt{\lambda}|x|^2
\]
for all $x\in\R$ with $1 \leq |x| \leq n^{1/(2\ell)}$ and all $n\in\N$.
\end{enumerate}
\end{lemma}

\begin{proof}
It is straightforward to check part (i) and part (ii).

We next prove part (iii). Let $\eta_1 \in (-\infty,\xi_1\wedge -1)$. Using the one-sided Lipschitz continuity of $\mu$ on the interval $(-\infty,\xi_1)$, we see that there exists $c_1 \in (0,\infty)$ such that for all $x \in (-\infty,\eta_1)$,
\[
\mu(x)-\mu(\eta_1) \geq c_1 \cdot (x-\eta_1) > c x,
\]
and therefore,
\begin{equation}\label{lambda1}
x \cdot \mu(x) \leq \big(c_1 + |\mu(\eta_1)|\big) \cdot |x|^2.
\end{equation}
Let $\eta_2 \in (\xi_k\vee 1,\infty)$. Similarly to \eqref{lambda1}, we obtain that there exists $c_2 \in (0,\infty)$ such that for all $x \in (\eta_2,\infty)$,
\begin{equation}\label{lambda2}
x \cdot \mu(x) \leq \big(c_2 + |\mu(\eta_2)|\big) \cdot |x|^2.
\end{equation}
Using the piecewise continuity of $\mu$, we see that there exists $c_3 \in (0,\infty)$ such that for all $x \in [\eta_1,\eta_2]$ with $|x| \geq 1$,
\begin{equation}\label{lambda3}
x \cdot \mu(x) \leq c_3 \cdot |x|^2.
\end{equation}
Combining \eqref{lambda1}, \eqref{lambda2} and \eqref{lambda3} completes the proof of part (iii).

We are now in a position to prove part (iv). Let $\eta_1 \in (-\infty,\xi_1\wedge -1)$. By Lemma \ref{assum2} it follows that there exists $c_1 \in (0,\infty)$ such that for all $x \in (-\infty,\eta_1)$ with $1 \leq |x| \leq n^{1/(2\ell)}$ and all $n\in\N$,
\begin{align*}
|\mu(x)-\mu(\eta_1)| 
&\leq c_1 \cdot \big(1+|x|^{\ell}+|\eta_1|^{\ell}\big) \cdot |x-\eta_1| \\
&\leq c_1 \cdot \big(1+2\sqrt{n}\big) \cdot |x|,
\end{align*}
and therefore,
\begin{align}\label{lambda4}
|\mu(x)| 
&\leq c_1 \cdot \big(1+2\sqrt{n}\big) \cdot |x| + |\mu(\eta_1)| \notag \\
&\leq 3c_1 \cdot \sqrt{n}|x| + |\mu(\eta_1)| \cdot \sqrt{n}|x| \notag \\
&= \big(3c_1 + |\mu(\eta_1)|\big) \cdot \sqrt{n}|x|.
\end{align}
Let $\eta_2 \in (\xi_k\vee 1,\infty)$. Similarly to \eqref{lambda4}, we may show that there exists $c_2 \in (0,\infty)$ such that for all $x \in (\eta_2,\infty)$ with $1 \leq |x| \leq n^{1/(2\ell)}$ and all $n\in\N$,
\begin{equation}\label{lambda5}
|\mu(x)| \leq \big(3c_2 + |\mu(\eta_2)|\big) \cdot \sqrt{n}|x|.
\end{equation}
Using the piecewise continuity of $\mu$, we see that there exists $c_3 \in (0,\infty)$ such that for all $x \in [\eta_1,\eta_2]$ with $1 \leq |x| \leq n^{1/(2\ell)}$ and all $n\in\N$,
\begin{equation}\label{lambda6}
|\mu(x)| \leq c_3 \cdot \sqrt{n}|x|.
\end{equation}
Combining \eqref{lambda4}, \eqref{lambda5} and \eqref{lambda6} completes the proof of part (iv).
\end{proof}

Next, choose $\lambda$ according to Lemma \ref{lambda}. Let
\begin{align*}
&\Delta W_{n,i} := W_{(i+1)/n} - W_{i/n}, \\
&\alpha_{n,i} := \ind_{\{ |\widehat{X}_{n,i/n}| \geq 1 \}} \frac{\sigma(\widehat{X}_{n,i/n})}{\widehat{X}_{n,i/n}} \Delta W_{n,i}
\end{align*}
for all $i \in \{ 0,1,\dots,n-1 \}$ and all $n\in\N$, let
\begin{align*}
&D_{n,i} := \big( \lambda + |x_0| \big) \cdot \exp \bigg( \lambda + \sup_{u \in \{ 0,1,\dots,i \}} \sum_{j=u}^{i-1} \big[ \lambda|\Delta W_{n,j}|^2 + \alpha_{n,j} \big] \bigg), \\
&\Omega_{n,i} := \Big\{ \omega \in \Omega : \sup_{j \in \{ 0,1,\dots,i-1 \}} D_{n,j} (\omega) \leq n^{1/(2\ell)}, 
\sup_{j \in \{ 0,1,\dots,i-1 \}} |\Delta W_{n,j}(\omega)| \leq 1 \Big\}
\end{align*}
for all $i \in \{ 0,1,\dots,n \}$ and all $n\in\N$, let
\[
\alpha_{n,i}^x := \ind_{\{ |\widehat{X}_{n,i/n}^x| \geq 1 \}} \frac{\sigma(\widehat{X}_{n,i/n}^x)}{\widehat{X}_{n,i/n}^x} \Delta W_{n,i}
\]
for all $x\in\R$, all $i \in \{ 0,1,\dots,n-1 \}$ and all $n\in\N$, and let
\begin{align*}
&D_{n,i}^x := \big( \lambda + |x| \big) \cdot \exp \bigg( \lambda + \sup_{u \in \{ 0,1,\dots,i \}} \sum_{j=u}^{i-1} \big[ \lambda|\Delta W_{n,j}|^2 + \alpha_{n,j}^x \big] \bigg), \\
&\Omega_{n,i}^x := \Big\{ \omega \in \Omega : \sup_{j \in \{ 0,1,\dots,i-1 \}} D_{n,j}^x (\omega) \leq n^{1/(2\ell)}, 
\sup_{j \in \{ 0,1,\dots,i-1 \}} |\Delta W_{n,j}(\omega)| \leq 1 \Big\}
\end{align*}
for all $x\in\R$, all $i \in \{ 0,1,\dots,n \}$ and all $n\in\N$.

\begin{lemma}
Let Assumption \ref{assum2} hold. Then we have that for all $i \in \{ 0,1,\dots,n \}$ and all $n\in\N$,
\begin{equation}\label{domin}
\ind_{\Omega_{n,i}} |\widehat{X}_{n,i/n}| \leq D_{n,i}.
\end{equation}
Furthermore, we have that for all $x\in\R$, all $i \in \{ 0,1,\dots,n \}$ and all $n\in\N$,
\[
\ind_{\Omega_{n,i}^x} |\widehat{X}_{n,i/n}^x| \leq D_{n,i}^x.
\]
\end{lemma}

\begin{proof}
We only prove \eqref{domin}. First of all, we have that
\[
|\Delta W_{n,i}| \leq 1
\]
on $\Omega_{n,i+1}$ for all $i \in \{ 0,1,\dots,n-1 \}$ and all $n\in\N$. Hence, by Lemma \ref{lambda}(i) we obtain that
\begin{align}\label{domin1}
|\widehat{X}_{n,(i+1)/n}| &\leq |\widehat{X}_{n,i/n}| + \frac1n |\mu(\widehat{X}_{n,i/n})| + |\sigma(\widehat{X}_{n,i/n})| |\Delta W_{n,i}| \notag \\
&\leq |\widehat{X}_{n,i/n}| + |\mu(\widehat{X}_{n,i/n})| + |\sigma(\widehat{X}_{n,i/n})| \notag \\
&\leq \lambda
\end{align}
on $\Omega_{n,i+1} \cap \{ \omega\in\Omega : |\widehat{X}_{n,i/n}(\omega)| \leq 1 \}$ for all $i \in \{ 0,1,\dots,n-1 \}$ and all $n\in\N$.

Moreover, using Lemma \ref{lambda} and the fact that for all $a,b\in\R$,
\begin{equation}\label{ine2}
ab \leq \frac{a^2+b^2}{2},
\end{equation}
we obtain that
\begin{align}\label{domin2}
|\widehat{X}_{n,(i+1)/n}|^2 &= \bigg| \widehat{X}_{n,i/n} + \frac{1/n \cdot \mu(\widehat{X}_{n,i/n})}{1 + 1/n \cdot |\mu(\widehat{X}_{n,i/n})|} + \sigma(\widehat{X}_{n,i/n}) \Delta W_{n,i} \bigg|^2 \notag \\
&= |\widehat{X}_{n,i/n}|^2 + \frac{|1/n \cdot \mu(\widehat{X}_{n,i/n})|^2}{(1 + 1/n \cdot |\mu(\widehat{X}_{n,i/n})|)^2} + |\sigma(\widehat{X}_{n,i/n}) \Delta W_{n,i}|^2 \notag \\
&\quad + \frac{2 \widehat{X}_{n,i/n} \cdot 1/n \cdot \mu(\widehat{X}_{n,i/n})}{1 + 1/n \cdot |\mu(\widehat{X}_{n,i/n})|} + 2 \widehat{X}_{n,i/n} \cdot \sigma(\widehat{X}_{n,i/n}) \Delta W_{n,i} \notag \\
&\quad + \frac{2/n \cdot \mu(\widehat{X}_{n,i/n}) \cdot \sigma(\widehat{X}_{n,i/n}) \Delta W_{n,i}}{1 + 1/n \cdot |\mu(\widehat{X}_{n,i/n})|} \notag \\
&\leq |\widehat{X}_{n,i/n}|^2 + \frac{1}{n^2} |\mu(\widehat{X}_{n,i/n})|^2 + |\sigma(\widehat{X}_{n,i/n})|^2 |\Delta W_{n,i}|^2 \notag \\
&\quad + \frac{2/n \cdot \widehat{X}_{n,i/n} \cdot \mu(\widehat{X}_{n,i/n})}{1 + 1/n \cdot |\mu(\widehat{X}_{n,i/n})|} + 2 \widehat{X}_{n,i/n} \cdot \sigma(\widehat{X}_{n,i/n}) \Delta W_{n,i} \notag \\
&\quad + \frac{2}{n} |\mu(\widehat{X}_{n,i/n})| |\sigma(\widehat{X}_{n,i/n})| |\Delta W_{n,i}| \notag \\
&\leq |\widehat{X}_{n,i/n}|^2 + \frac{2}{n^2} |\mu(\widehat{X}_{n,i/n})|^2 + 2 |\sigma(\widehat{X}_{n,i/n})|^2 |\Delta W_{n,i}|^2 \notag \\
&\quad + \frac{2/n \cdot \widehat{X}_{n,i/n} \cdot \mu(\widehat{X}_{n,i/n})}{1 + 1/n \cdot |\mu(\widehat{X}_{n,i/n})|} + 2 \widehat{X}_{n,i/n} \cdot \sigma(\widehat{X}_{n,i/n}) \Delta W_{n,i} \notag \\
&\leq |\widehat{X}_{n,i/n}|^2 + \frac{2\sqrt{\lambda}}{n} |\widehat{X}_{n,i/n}|^2 + 2 \lambda |\widehat{X}_{n,i/n}|^2 |\Delta W_{n,i}|^2 \notag \\
&\quad + \frac{2\sqrt{\lambda}}{n} |\widehat{X}_{n,i/n}|^2 + 2\widehat{X}_{n,i/n} \cdot \sigma(\widehat{X}_{n,i/n}) \Delta W_{n,i} \notag \\
&= |\widehat{X}_{n,i/n}|^2 \bigg(1 + \frac{4\sqrt{\lambda}}{n} + 2\lambda |\Delta W_{n,i}|^2 + 2\alpha_{n,i}\bigg) \notag \\
&\leq |\widehat{X}_{n,i/n}|^2 \exp\bigg(\frac{2\lambda}{n} + 2\lambda |\Delta W_{n,i}|^2 + 2\alpha_{n,i}\bigg)
\end{align}
on $\{ \omega\in\Omega : 1 \leq |\widehat{X}_{n,i/n}(\omega)| \leq n^{1/(2\ell)} \}$ for all $i \in \{ 0,1,\dots,n-1 \}$ and all $n\in\N$.

Additionally, we use the mappings $\tau_{n,l} : \Omega \to \{ -1,0,1,\dots,l-1 \}, l \in \{ 0,1,\dots,n \}, n\in\N$, given by
\[
\tau_{n,l}(\omega) := \max(\{ -1 \} \cup \{ i \in \{ 0,1,\dots,l-1 \} : |\widehat{X}_{n,i/n}(\omega)| \leq 1 \})
\]
for all $\omega\in\Omega$, all $l \in \{ 0,1,\dots,n \}$ and all $n\in\N$.

We prove by induction on $i \in \{ 0,1,\dots,n \}$ where $n\in\N$ is fixed. The base case $i = 0$ is trivial. Now let $l \in \{ 0,1,\dots,n-1 \}$ be arbitrary and assume that \eqref{domin} holds for all $i \in \{ 0,1,\dots,l \}$. We then show \eqref{domin} for $i = l+1$. Let $\omega\in\Omega_{n,l+1}$ be arbitrary. From the induction hypothesis and from $\omega\in\Omega_{n,l+1}\subset\Omega_{n,i+1}$ it follows that
\[
|\widehat{X}_{n,i/n}(\omega)| \leq D_{n,i}(\omega) \leq n^{1/(2\ell)}
\]
for all $i \in \{ 0,1,\dots,l \}$. We therefore obtain that
\[
1 \leq |\widehat{X}_{n,i/n}(\omega)| \leq n^{1/(2\ell)}
\]
for all $i \in \{ \tau_{n,l+1}(\omega)+1,\tau_{n,l+1}(\omega)+2,\dots,l \}$. Estimate \eqref{domin2} thus gives that
\begin{equation}\label{domin3}
|\widehat{X}_{n,(i+1)/n}(\omega)| \leq |\widehat{X}_{n,i/n}(\omega)| \cdot \exp\bigg(\frac{\lambda}{n} + \lambda |\Delta W_{n,i}(\omega)|^2 + \alpha_{n,i}(\omega)\bigg)
\end{equation}
for all $i \in \{ \tau_{n,l+1}(\omega)+1,\tau_{n,l+1}(\omega)+2,\dots,l \}$. Iterating \eqref{domin3} hence yields that
\begin{align*}
&|\widehat{X}_{n,(l+1)/n}(\omega)| \\
&\leq |\widehat{X}_{n,l/n}(\omega)| \cdot \exp\bigg(\frac{\lambda}{n} + \lambda |\Delta W_{n,l}(\omega)|^2 + \alpha_{n,l}(\omega)\bigg) \\
&\leq \cdots \leq |\widehat{X}_{n,(\tau_{n,l+1}(\omega)+1)/n}(\omega)| \cdot \exp\bigg(\sum_{i=\tau_{n,l+1}(\omega)+1}^l \bigg(\frac{\lambda}{n} + \lambda |\Delta W_{n,i}(\omega)|^2 + \alpha_{n,i}(\omega)\bigg)\bigg) \\
&\leq |\widehat{X}_{n,(\tau_{n,l+1}(\omega)+1)/n}(\omega)| \cdot \exp\bigg(\lambda + \sup_{u \in \{ 0,1,\dots,l+1 \}} \sum_{i=u}^l \big(\lambda |\Delta W_{n,i}(\omega)|^2 + \alpha_{n,i}(\omega)\big)\bigg).
\end{align*}
Estimate \eqref{domin1} therefore shows that
\begin{align*}
|\widehat{X}_{n,(l+1)/n}(\omega)| 
&\leq \big(\lambda + |x_0(\omega)|\big) \cdot \exp\bigg(\lambda + \sup_{u \in \{ 0,1,\dots,l+1 \}} \sum_{i=u}^l \big(\lambda |\Delta W_{n,i}(\omega)|^2 + \alpha_{n,i}(\omega)\big)\bigg) \\
&\leq D_{n,l+1}(\omega).
\end{align*}
This finishes the induction step and the proof.
\end{proof}

\begin{lemma}
Let $p \in [1,\infty)$. Then we have that
\[
\sup_{\substack{n\in\N, \\ n \geq 4 \lambda p}} \EE\bigg[ \exp\bigg( p\lambda \sum_{j=0}^{n-1} |\Delta W_{n,j}|^2 \bigg) \bigg] < \infty.
\]
\end{lemma}

\begin{proof}
See the proof of \cite[Lemma 3.3]{MH}.
\end{proof}

\begin{lemma}
Let Assumption \ref{assum2} hold and $p \in [1,\infty)$. Then we have that
\[
\sup_{z \in \{-1,1\}} \sup_{n\in\N} \EE\bigg[ \sup_{i \in \{ 0,1,\dots,n \}} \exp\bigg( pz \sum_{j=0}^{i-1} \alpha_{n,j} \bigg) \bigg] < \infty.
\]
Furthermore, there exists $c \in (0,\infty)$ such that for all $x\in\R$,
\[
\sup_{z \in \{-1,1\}} \sup_{n\in\N} \EE\bigg[ \sup_{i \in \{ 0,1,\dots,n \}} \exp\bigg( pz \sum_{j=0}^{i-1} \alpha_{n,j}^x \bigg) \bigg] \leq c.
\]
\end{lemma}

\begin{proof}
See the proof of \cite[Lemma 3.4]{MH}.
\end{proof}

\begin{lemma}
Let Assumption \ref{assum2} hold and $p \in [1,\infty)$. Then we have that
\[
\sup_{\substack{n\in\N, \\ n \geq 8 \lambda p}} \EE\Big[ \sup_{n \in \{ 0,1,\dots,n \}} |D_{n,i}|^p \Big] < \infty.
\]
Furthermore, there exists $c \in (0,\infty)$ such that for all $x\in\R$,
\[
\sup_{\substack{n\in\N, \\ n \geq 8 \lambda p}} \Big(\EE\Big[ \sup_{n \in \{ 0,1,\dots,n \}} |D_{n,i}^x|^p \Big]\Big)^{1/p} \leq c \cdot \big( 1+|x| \big).
\]
\end{lemma}

\begin{proof}
See the proof of \cite[Lemma 3.5]{MH}.
\end{proof}

\begin{lemma}
Let Assumption \ref{assum2} hold and $p \in [1,\infty)$. Then we have that
\[
\sup_{n\in\N} \big( n^p \cdot \PP [(\Omega_{n,n})^{c}] \big) < \infty.
\]
Furthermore, there exists $c \in (0,\infty)$ such that for all $x\in\R$,
\[
\sup_{n\in\N} \big( n^p \cdot \PP [(\Omega_{n,n}^x)^{c}] \big)^{1/p} \leq c \cdot \big( 1 + |x|^{4(1+\ell)} \big).
\]
\end{lemma}

\begin{proof}
See the proof of \cite[Lemma 3.6]{MH}.
\end{proof}

\begin{lemma}\label{tamprop0}
Let Assumption \ref{assum2} hold and $p \in [1,\infty)$. Then we have that
\[
\sup_{n\in\N} \sup_{i \in \{ 0,1,\dots,n \}} \EE\big[ |\widehat{X}_{n,i/n}|^p \big] < \infty.
\]
Furthermore, there exists $c \in (0,\infty)$ such that for all $x\in\R$,
\[
\sup_{n\in\N} \sup_{i \in \{ 0,1,\dots,n \}} \big(\EE\big[ |\widehat{X}_{n,i/n}^x|^p \big]\big)^{1/p} \leq c \cdot \big(1+|x|^{5(1+\ell)}\big).
\]
\end{lemma}

\begin{proof}
See the proof of \cite[Lemma 3.9]{MH}.
\end{proof}

\begin{lemma}\label{tamprop1}
Let Assumption \ref{assum2} hold and $p \in [1,\infty)$. Then there exists $c \in (0,\infty)$ such that for all $n\in\N$, all $\delta \in [0,1]$ and all $t \in [0,1-\delta]$,
\[
\Big(\EE\Big[ \sup_{s \in [t,t+\delta]} |\widehat{X}_{n,s}-\widehat{X}_{n,t}|^p \Big]\Big)^{1/p} \leq c \cdot \sqrt{\delta}.
\]
In particular,
\[
\sup_{n\in\N} \EE\big[ \|\widehat{X}_n\|_{\infty}^p \big] < \infty.
\]
\end{lemma}

\begin{proof}
Without loss of generality we may assume that $p \in [2,\infty)$. It follows from \eqref{ine} that
\begin{align*}
&\EE\Big[ \sup_{s \in [t,t+\delta]} |\widehat{X}_{n,s}-\widehat{X}_{n,t}|^p \Big] \\
&\leq 2^{p-1} \EE\bigg[ \sup_{s \in [t,t+\delta]} \bigg|\int_t^s \frac{\mu(\widehat{X}_{n,\urn})}{1 + 1/n \cdot |\mu(\widehat{X}_{n,\urn})|} \, \dif r\bigg|^p \bigg] 
+ 2^{p-1} \EE\bigg[ \sup_{s \in [t,t+\delta]} \bigg|\int_t^s \sigma(\widehat{X}_{n,\urn}) \, \dif W_r\bigg|^p \bigg].
\end{align*}
Hence, using the H\"{o}lder inequality, \cite[Theorem 1.7.2]{XM} and Assumption \ref{assum2} we obtain that there exists $c \in (0,\infty)$ such that
\begin{align*}
&\EE\Big[ \sup_{s \in [t,t+\delta]} |\widehat{X}_{n,s}-\widehat{X}_{n,t}|^p \Big] \\
&\leq (2\delta)^{p-1} \EE\bigg[ \int_t^{t+\delta} |\mu(\widehat{X}_{n,\urn})|^p \, \dif r \bigg] 
+ \frac12 \bigg(\frac{2 p^3}{p-1}\bigg)^{p/2} \delta^{(p-2)/2} \EE\bigg[ \int_t^{t+\delta} |\sigma(\widehat{X}_{n,\urn})|^p \, \dif r \bigg] \\
&\leq c \delta^p \cdot \Big(1 + \sup_{t \in [0,1]} \EE\big[ |\widehat{X}_{n,\utn}|^{(1+\ell)p} \big]\Big) 
+ c \delta^{p/2} \cdot \Big(1 + \sup_{t \in [0,1]} \EE\big[ |\widehat{X}_{n,\utn}|^p \big]\Big),
\end{align*}
which, together with Lemma \ref{tamprop0}, completes the proof.
\end{proof}

\begin{lemma}\label{tamprop2}
Let Assumption \ref{assum2} hold and $p \in [1,\infty)$. Then there exists $c \in (0,\infty)$ such that for all $x\in\R$, all $n\in\N$, all $\delta \in [0,1]$ and all $t \in [0,1-\delta]$,
\[
\Big(\EE\Big[ \sup_{s \in [t,t+\delta]} |\widehat{X}_{n,s}^x-\widehat{X}_{n,t}^x|^p \Big]\Big)^{1/p} \leq c \cdot \big(1+|x|^{5(1+\ell)^2}\big) \cdot \sqrt{\delta}.
\]
In particular,
\[
\sup_{n\in\N} \big(\EE\big[ \|\widehat{X}_n^x\|_{\infty}^p \big]\big)^{1/p} \leq c \cdot \big(1+|x|^{5(1+\ell)^2}\big).
\]
\end{lemma}

\begin{proof}
See the proof of Lemma \ref{tamprop1}.
\end{proof}

\subsection{A Markov property and occupation time estimates for the time-continuous tamed Euler scheme}

The following lemma provides a Markov property of the time-continuous tamed Euler scheme $\widehat{X}_n$ relative to the gridpoints $1/n,2/n,\dots,1$.

\begin{lemma}\label{markov}
For all $n\in\N$, all $j \in \{ 0,\dots,n-1 \}$ and $\PP^{\widehat{X}_{n,j/n}}$-almost all $x\in\R$ we have
\[
\PP^{(\widehat{X}_{n,t})_{t \in [j/n,1]} | \F_{j/n}} = \PP^{(\widehat{X}_{n,t})_{t \in [j/n,1]} | \widehat{X}_{n,j/n}}
\]
as well as
\[
\PP^{(\widehat{X}_{n,t})_{t \in [j/n,1]} | \widehat{X}_{n,j/n} = x} = \PP^{(\widehat{X}_{n,t}^x)_{t \in [0,1-j/n]}}.
\]
\end{lemma}

\begin{proof}
See the proof of \cite[Lemma 3]{TM}.
\end{proof}

Next, we provide an estimate for the expected occupation time of a neighborhood of a non-zero of $\sigma$ by the time-continuous tamed Euler scheme $\widehat{X}_n^x$.

\begin{lemma}\label{occup}
Let Assumption \ref{assum2} hold. Let $\xi\in\R$ satisfy $\sigma(\xi) \neq 0$. Then there exists $c \in (0,\infty)$ such that for all $x\in\R$, all $n\in\N$ and all $\eps \in (0,\infty)$,
\[
\int_0^1 \PP\big(\big\{ |\widehat{X}_{n,t}^x-\xi| \leq \eps \big\}\big) \, \dif t 
\leq c \cdot \big(1+|x|^{10(1+\ell)^2}\big) \cdot \bigg(\eps+\frac{1}{\sqrt{n}}\bigg).
\]
\end{lemma}

\begin{proof}
Let $x\in\R$ and $n\in\N$. By \eqref{intrep}, Assumption \ref{assum2} and Lemma \ref{tamprop2} we see that $\widehat{X}_n^x$ is a continuous semi-martingale with quadratic variation
\begin{equation}\label{qv}
\langle \widehat{X}_n^x \rangle_t = \int_0^t \sigma^2(\widehat{X}_{n,\usn}^x) \, \dif s, \quad t \in [0,1].
\end{equation}
For $a\in\R$, let $L^a(\widehat{X}_n^x) = (L_t^a(\widehat{X}_n^x))_{t \in [0,1]}$ denote the local time of $\widehat{X}_n^x$ at the point $a$. Thus, for all $a\in\R$ and all $t \in [0,1]$,
\begin{align*}
|\widehat{X}_{n,t}^x-a| &= |x-a| + \int_0^t \sgn(\widehat{X}_{n,s}^x-a) \cdot \frac{\mu(\widehat{X}_{n,\usn}^x)}{1 + 1/n \cdot | \mu(\widehat{X}_{n,\usn}^x) |} \, \dif s \\
&\quad + \int_0^t \sgn(\widehat{X}_{n,s}^x-a) \cdot \sigma(\widehat{X}_{n,\usn}^x) \, \dif W_s + L_t^a(\widehat{X}_n^x),
\end{align*}
where $\sgn(z) = \ind_{(0,\infty)}(z) - \ind_{(-\infty,0]}(z)$ for $z\in\R$, see, e.g. \cite[Chap. VI]{DR}. Hence, for all $a\in\R$ and all $t \in [0,1]$,
\[
L_t^a(\widehat{X}_n^x) \leq |\widehat{X}_{n,t}^x-x| + \int_0^t |\mu(\widehat{X}_{n,\usn}^x)| \, \dif s + \bigg| \int_0^t \sgn(\widehat{X}_{n,s}^x-a) \cdot \sigma(\widehat{X}_{n,\usn}^x) \, \dif W_s \bigg|.
\]
  
Using the H\"{o}lder inequality, the Burkholder--Davis--Gundy inequality, Assumption \ref{assum2}, Lemma \ref{tamprop0} and \ref{tamprop2} we conclude that there exists $c \in (0,\infty)$ such that for all $x\in\R$, all $n\in\N$, all $a\in\R$ and all $t \in [0,1]$,
\begin{equation}\label{local1}
\EE\big[L_t^a(\widehat{X}_n^x)\big] \leq c \cdot \big(1+|x|^{5(1+\ell)^2}\big).
\end{equation}
Using \eqref{qv} and \eqref{local1} we obtain by the occupation time formula that there exists $c \in (0,\infty)$ such that for all $x\in\R$, all $n\in\N$ and all $\eps \in (0,\infty)$,
\begin{equation}\label{local2}
\EE\bigg[ \int_0^1 \ind_{[\xi-\eps,\xi+\eps]}(\widehat{X}_{n,t}^x) \cdot \sigma^2(\widehat{X}_{n,\utn}^x) \, \dif t \bigg]
= \int_{\R} \ind_{[\xi-\eps,\xi+\eps]}(a) \cdot \EE\big[L_1^a(\widehat{X}_n^x)\big] \, \dif a
\leq c \cdot \big(1+|x|^{5(1+\ell)^2}\big) \cdot \eps.
\end{equation}
  
Using the Lipschitz continuity of $\sigma$ and Lemma \ref{tamprop2} we obtain that there exist $c_1,c_2 \in (0,\infty)$ such that for all $x\in\R$ and all $n\in\N$,
\begin{align}\label{local3}
\EE\bigg[ \int_0^1 |\sigma^2(\widehat{X}_{n,t}^x)-\sigma^2(\widehat{X}_{n,\utn}^x)| \, \dif t \bigg]
&\leq c_1 \cdot \int_0^1 \EE\big[ |\widehat{X}_{n,t}^x-\widehat{X}_{n,\utn}^x| \cdot \big(1+\|\widehat{X}_n^x\|_{\infty}\big) \big] \, \dif t \notag \\
&\leq c_2 \cdot \big(1+|x|^{10(1+\ell)^2}\big) \cdot \frac{1}{\sqrt{n}}.
\end{align}
  
Since $\sigma$ is continuous and $\sigma(\xi) \neq 0$, there exist $\kappa,\eps_0 \in (0,\infty)$ such that
\[
\inf_{|z-\xi| < \eps_0} \sigma^2(z) \geq \kappa.
\]
Observing \eqref{local2} and \eqref{local3}, we conclude that there exists $c \in (0,\infty)$ such that for all $x\in\R$, all $n\in\N$ and all $\eps \in (0,\eps_0]$,
\begin{align*}
&\int_0^1 \PP\big(\big\{ |\widehat{X}_{n,t}^x-\xi| \leq \eps \big\}\big) \, \dif t \\
&= \frac{1}{\kappa} \cdot \EE\bigg[ \int_0^1 \kappa \cdot \ind_{[\xi-\eps,\xi+\eps]}(\widehat{X}_{n,t}^x) \, \dif t \bigg] \\
&\leq \frac{1}{\kappa} \cdot \EE\bigg[ \int_0^1 \ind_{[\xi-\eps,\xi+\eps]}(\widehat{X}_{n,t}^x) \cdot \sigma^2(\widehat{X}_{n,t}^x) \, \dif t \bigg] \\
&\leq \frac{1}{\kappa} \cdot \EE\bigg[ \int_0^1 \big(\ind_{[\xi-\eps,\xi+\eps]}(\widehat{X}_{n,t}^x) \cdot \sigma^2(\widehat{X}_{n,\utn}^x) + |\sigma^2(\widehat{X}_{n,t}^x) - \sigma^2(\widehat{X}_{n,\utn}^x)|\big) \, \dif t \bigg] \\
&\leq \frac{c}{\kappa} \cdot \big(1+|x|^{5(1+\ell)^2}+|x|^{10(1+\ell)^2}\big) \cdot \bigg(\eps + \frac{1}{\sqrt{n}}\bigg),
\end{align*}
which completes the proof.
\end{proof}

\begin{lemma}\label{central00}
Let Assumption \ref{assum2} hold. Then there exists $c \in (0,\infty)$ such that for all $x\in\R$ and all $n\in\N$,
\[
\frac{|\mu(x)|}{1 + 1/n \cdot |\mu(x)|} \leq c \cdot n^{1-\gamma} \cdot \big(1+|x|\big),
\]
where $\gamma = \frac{1}{1+\ell} \wedge \frac12$.
\end{lemma}

\begin{proof}
Using Assumption \ref{assum2} and the fact that for all $a, b \in (0,\infty)$,
\[
a+b \geq \frac{a^{\gamma}b^{1-\gamma}}{\gamma^{\gamma}(1-\gamma)^{1-\gamma}},
\]
we conclude that there exist $c_1,c_2 \in (0,\infty)$ such that for all $x\in\R$ and all $n\in\N$,
\begin{align*}
\frac{|\mu(x)|}{1 + 1/n \cdot |\mu(x)|} 
&\leq \frac{c_1 \cdot \big(1+|x|\big)^{1+\ell}}{1 + 1/n \cdot c_1 \cdot \big(1+|x|\big)^{1+\ell}} \\
&= \frac{c_1 \cdot n \cdot \big(1+|x|\big)}{n \cdot \big(1+|x|\big)^{-\ell} + c_1 \cdot \big(1+|x|\big)} \\
&\leq \frac{c_2 \cdot n \cdot \big(1+|x|\big)}{n^{\gamma} \cdot \big(1+|x|\big)^{-\ell\gamma} \cdot \big(1+|x|\big)^{1-\gamma}} \\
&= c_2 \cdot n^{1-\gamma} \cdot \big(1+|x|\big)^{(1+\ell)\gamma} \\
&\leq c_2 \cdot n^{1-\gamma} \cdot \big(1+|x|\big),
\end{align*}
which completes the proof.
\end{proof}

The following result shows how to transfer the condition of a sign change of $\widehat{X}_n-\xi$ at time $t$ relative to its sign at the grid point $\utn$ to a condition on the distance of $\widehat{X}_n$ and $\xi$ at the time $\utn-(t-\utn)$.

\begin{lemma}\label{central}
Let Assumption \ref{assum2} hold. Let $\xi\in\R$. Then there exists $c \in (0,\infty)$ such that for all $n\in\N$, all $0 \leq s \leq t \leq 1$ with $\utn - s \geq 1/n$ and all $A\in\F_s$,
\begin{align}\label{central01}
&\PP\big(A \cap \big\{ (\widehat{X}_{n,t}-\xi) \cdot (\widehat{X}_{n,\utn}-\xi) \leq 0 \big\}\big) \notag \\
&\leq \frac{c}{n} \cdot \PP(A)
+ c \cdot \int_{\R} \PP\bigg(A \cap \bigg\{ |\widehat{X}_{n,\utn-(t-\utn)}-\xi| \leq \frac{c}{\sqrt{n}} \big(1+|z|\big) \bigg\}\bigg) \cdot e^{-\frac{z^2}{2}} \, \dif z.
\end{align}
\end{lemma}

\begin{proof}
Let $\gamma = \frac{1}{1+\ell} \wedge \frac12$. According to Assumption \ref{assum2} and Lemma \ref{central00}, choose $K \in (0,\infty)$ such that for all $x\in\R$ and all $n\in\N$,
\begin{equation}\label{central02}
n^{\gamma-1} \cdot \frac{|\mu(x)|}{1 + 1/n \cdot |\mu(x)|} + |\sigma(x)| \leq K \cdot \big(1+|x|\big),
\end{equation}
and for all $x\in\R$ with $|x-\xi| \leq 3$,
\begin{equation}\label{central03}
|\mu(x)| + |\sigma(x)| \leq K.
\end{equation}
Choose $n_0 \in \N \setminus \{ 1 \}$ such that for all $n \geq n_0$,
\[
2K \cdot \big(1+|\xi|\big) \cdot \frac{1 + \sqrt{2\ln(n)}}{n^{\gamma}} \leq \frac12.
\]
Without loss of generality we may assume that $n \geq n_0$. Let $0 \leq s \leq t \leq 1$ with $\utn-s \geq 1/n$ and let $A\in\F_s$. If $t = \utn$, then for all $c \in (0,\infty)$ and all $z\in\R$, we have
\[
\big\{ (\widehat{X}_{n,t}-\xi) \cdot (\widehat{X}_{n,\utn}-\xi) \leq 0 \big\}
= \{ \widehat{X}_{n,\utn}-\xi = 0 \}
\subset \bigg\{ |\widehat{X}_{n,\utn-(t-\utn)}-\xi| \leq \frac{c}{\sqrt{n}} \big(1+|z|\big) \bigg\},
\]
which implies that \eqref{central01} holds for all $c \geq 1/\sqrt{2\pi}$.
  
If $t > \utn$, then put
\[
Z_1 = \frac{W_{t}-W_{\utn}}{\sqrt{t-\utn}}, \quad
Z_2 = \frac{W_{\utn}-W_{\utn-(t-\utn)}}{\sqrt{t-\utn}}, \quad
Z_3 = \frac{W_{\utn-(t-\utn)}-W_{\utn-1/n}}{\sqrt{1/n-(t-\utn)}}.
\]
We will show below
\begin{align}\label{central04}
&\big\{ (\widehat{X}_{n,t}-\xi) \cdot (\widehat{X}_{n,\utn}-\xi) \leq 0 \big\} \cap \Big\{ \max_{i \in \{ 1,2,3 \}} |Z_i| \leq \sqrt{2\ln(n)} \Big\} \notag \\
&\subset \big\{ |\widehat{X}_{n,\utn-(t-\utn)}-\xi| \leq K \cdot \big(2+|Z_1|+|Z_2|\big) / \sqrt{n} \big\}.
\end{align}
Note that $Z_1,Z_2,Z_3$ are independent and identically distributed standard normal random variables. Moreover, $(Z_1,Z_2,Z_3)$ is independent of $\F_s$ since $s \leq \utn-1/n$, $(Z_1,Z_2)$ is independent of $\F_{\utn-(t-\utn)}$ and $\widehat{X}_{n,\utn-(t-\utn)}$ is $\F_{\utn-(t-\utn)}$-measurable. Using the latter facts, together with \eqref{central04} and a standard estimate of standard normal tail probabilities, we obtain that
\begin{align*}
&\PP\big(A \cap \big\{ (\widehat{X}_{n,t}-\xi) \cdot (\widehat{X}_{n,\utn}-\xi) \leq 0 \big\}\big) \\
&\leq \PP\big(A \cap \big\{ |\widehat{X}_{n,\utn-(t-\utn)}-\xi| \leq K \cdot \big(2+|Z_1|+|Z_2|\big) / \sqrt{n} \big\}\big) \\
&\quad + \PP\Big( A \cap \Big\{ \max_{i \in \{ 1,2,3 \}} |Z_i| > \sqrt{2\ln(n)} \Big\}\Big) \\
&\leq \frac{2}{\pi} \int_{[0,\infty)^2} \PP\bigg( A \cap \bigg\{ |\widehat{X}_{n,\utn-(t-\utn)}-\xi| \leq K \cdot \frac{2+z_1+z_2}{\sqrt{n}} \bigg\}\bigg) \cdot e^{-\frac{z_1^2+z_2^2}{2}} \, \dif (z_1,z_2) \\
&\quad + 6\PP(A) \cdot \PP(\{ Z_1 > \sqrt{2\ln(n)} \}) \\
&\leq \frac{2}{\pi} \int_{\R^{2}} \PP\bigg( A \cap \bigg\{ |\widehat{X}_{n,\utn-(t-\utn)}-\xi| \leq 2K \cdot \frac{1+|\frac{z_1+z_2}{\sqrt{2}}|}{\sqrt{n}} \bigg\}\bigg) \cdot e^{-\frac{z_1^2+z_2^2}{2}} \, \dif (z_1,z_2) \\
&\quad + \frac{6\PP(A)}{\sqrt{2\pi \cdot 2\ln(n)} \cdot n} \\
&= \frac{4}{\sqrt{2\pi}} \int_{\R} \PP\bigg( A \cap \bigg\{ |\widehat{X}_{n,\utn-(t-\utn)}-\xi| \leq 2K \cdot \frac{1+|z|}{\sqrt{n}} \bigg\}\bigg) \cdot e^{-\frac{z^2}{2}} \, \dif z
+ \frac{3\PP(A)}{\sqrt{\pi\ln(n)} \cdot n},
\end{align*}
which yields \eqref{central01}.
  
It remains to prove the inclusion \eqref{central04}. To this end let $\omega \in \Omega$ and assume that
\[
(\widehat{X}_{n,t}(\omega)-\xi) \cdot (\widehat{X}_{n,\utn}(\omega)-\xi) \leq 0 \quad \textrm{and} \quad \max_{i \in \{ 1,2,3 \}} |Z_i(\omega)| \leq \sqrt{2\ln(n)}.
\]
Using \eqref{central02} and the fact that for all $a,b\in\R$,
\begin{equation}\label{central05}
1+|a| \leq \big(1+|a-b|\big) \cdot \big(1+|b|\big),
\end{equation}
we obtain
\begin{align}\label{central06}
|\widehat{X}_{n,\utn}(\omega)-\xi|
&\leq |(\widehat{X}_{n,\utn}(\omega)-\xi) - (\widehat{X}_{n,t}(\omega)-\xi)| \notag \\
&= \bigg| \frac{\mu(\widehat{X}_{n,\utn}(\omega))}{1 + 1/n \cdot |\mu(\widehat{X}_{n,\utn}(\omega))|} \cdot (t-\utn) + \sigma(\widehat{X}_{n,\utn}(\omega)) \cdot \sqrt{t-\utn} \cdot Z_1(\omega) \bigg| \notag \\
&\leq K \cdot \big(1+|\widehat{X}_{n,\utn}(\omega)|\big) \cdot \bigg( \frac{1}{n^{\gamma}} + \frac{1}{\sqrt{n}} \cdot |Z_1(\omega)| \bigg) \notag \\
&\leq \big(1+|\widehat{X}_{n,\utn}(\omega)-\xi|\big) \cdot K \cdot \big(1+|\xi|\big) \cdot \frac{1}{n^{\gamma}} \cdot \big(1+|Z_1(\omega)|\big).
\end{align}
Since $n \geq n_0$ we have
\[
K \cdot \big(1+|\xi|\big) \cdot \frac{1}{n^{\gamma}} \cdot \big(1+|Z_1(\omega)|\big) \leq K \cdot \big(1+|\xi|\big) \cdot \frac{1+\sqrt{2\ln(n)}}{n^{\gamma}} \leq \frac12,
\]
and therefore,
\begin{equation}\label{central07}
|\widehat{X}_{n,\utn}(\omega)-\xi| \leq 1.
\end{equation}
Similarly to \eqref{central06}, we obtain by \eqref{central02} and \eqref{central05} that
\begin{equation}\label{central08}
|\widehat{X}_{n,\utn}(\omega)-\widehat{X}_{n,\utn-(t-\utn)}(\omega)|
\leq \big(1 + |\widehat{X}_{n,\utn-1/n}(\omega)-\xi|\big) \cdot K \cdot \big(1+|\xi|\big) \cdot \frac{1}{n^{\gamma}} \cdot \big(1+|Z_2(\omega)|\big)
\end{equation}
and
\begin{equation}\label{central09}
|\widehat{X}_{n,\utn-(t-\utn)}(\omega) - \widehat{X}_{n,\utn-1/n}(\omega)|
\leq \big(1 + |\widehat{X}_{n,\utn-1/n}(\omega)-\xi|\big) \cdot K \cdot \big(1+|\xi|\big) \cdot \frac{1}{n^{\gamma}} \cdot \big(1+|Z_3(\omega)|\big).
\end{equation}
Using \eqref{central08} and \eqref{central09} we obtain
\[
|\widehat{X}_{n,\utn}(\omega) - \widehat{X}_{n,\utn-1/n}(\omega)|
\leq \big(1 + |\widehat{X}_{n,\utn-1/n}(\omega)-\xi|\big) \cdot K \cdot \big(1+|\xi|\big) \cdot \frac{1}{n^{\gamma}} \cdot \big(2+|Z_2(\omega)|+|Z_3(\omega)|\big).
\]
Since $n \geq n_0$ we have
\[
K \cdot \big(1+|\xi|\big) \cdot \frac{1}{n^{\gamma}} \cdot \big(2+|Z_2(\omega)|+|Z_3(\omega)|\big) \leq 2K \cdot \big(1+|\xi|\big) \cdot \frac{1 + \sqrt{2\ln(n)}}{n^{\gamma}} \leq \frac12,
\]
and therefore,
\begin{equation}\label{central10}
|\widehat{X}_{n,\utn}(\omega) - \widehat{X}_{n,\utn-1/n}(\omega)| 
\leq \big(1 + |\widehat{X}_{n,\utn-1/n}(\omega)-\xi|\big) / 2.
\end{equation}
Using \eqref{central07} and \eqref{central10} we obtain
\begin{align*}
|\widehat{X}_{n,\utn-1/n}(\omega) - \xi|
&\leq |\widehat{X}_{n,\utn}(\omega) - \widehat{X}_{n,\utn-1/n}(\omega)| + |\widehat{X}_{n,\utn}(\omega) - \xi| \\
&\leq \big(1 + |\widehat{X}_{n,\utn-1/n}(\omega) - \xi|\big) / 2 + 1,
\end{align*}
and therefore,
\begin{equation}\label{central11}
|\widehat{X}_{n,\utn-1/n}(\omega) - \xi| \leq 3.
\end{equation}
Using \eqref{central03} and \eqref{central07} we obtain
\begin{align}\label{central12}
|\widehat{X}_{n,\utn}(\omega) - \xi|
&\leq |(\widehat{X}_{n,\utn}(\omega)-\xi) - (\widehat{X}_{n,t}(\omega)-\xi)| \notag \\
&= \bigg| \frac{\mu(\widehat{X}_{n,\utn}(\omega))}{1 + 1/n \cdot |\mu(\widehat{X}_{n,\utn}(\omega))|} \cdot (t-\utn) + \sigma(\widehat{X}_{n,\utn}(\omega)) \cdot \sqrt{t-\utn} \cdot Z_1(\omega) \bigg| \notag \\
&\leq K \cdot \bigg( \frac{1}{n} + \frac{1}{\sqrt{n}} \cdot |Z_1(\omega)| \bigg) \notag \\
&\leq K \cdot \frac{1}{\sqrt{n}} \cdot \big(1 + |Z_1(\omega)|\big).
\end{align}
Similarly to \eqref{central12}, we obtain by \eqref{central03} and \eqref{central11} that
\begin{equation}\label{central13}
|\widehat{X}_{n,\utn}(\omega) - \widehat{X}_{n,\utn-(t-\utn)}(\omega)|
\leq K \cdot \frac{1}{\sqrt{n}} \cdot \big(1 + |Z_2(\omega)|\big).
\end{equation}
Using \eqref{central12} and \eqref{central13} we obtain
\begin{align*}
|\widehat{X}_{n,\utn-(t-\utn)}(\omega) - \xi|
&\leq |\widehat{X}_{n,\utn}(\omega) - \widehat{X}_{n,\utn-(t-\utn)}(\omega)| + |\widehat{X}_{n,\utn}(\omega) - \xi| \\
&\leq K \cdot \frac{1}{\sqrt{n}} \cdot \big(2 + |Z_1(\omega)| + |Z_2(\omega)|\big).
\end{align*}
This finishes the proof of \eqref{central04}.
\end{proof}

Using Lemmas \ref{markov}, \ref{occup} and \ref{central} we can now establish the following two estimates on the probability of sign changes of $\widehat{X}_n-\xi$ relative to its sign at the gridpoints $0,1/n,\dots,1$.

\begin{lemma} \label{key}
Let Assumption \ref{assum2} hold. Let $\xi\in\R$ satisfy $\sigma(\xi) \neq 0$, $\alpha \in [2,\infty)$ and
\[
A_{n,t} = \big\{ (\widehat{X}_{n,t}-\xi) \cdot (\widehat{X}_{n,\utn}-\xi) \leq 0 \big\}
\]
for all $n\in\N$ and all $t \in [0,1]$. Then the following two statements hold.
\begin{enumerate}[\em (i)]
\item There exists $c \in (0,\infty)$ such that for all $n\in\N$, all $s \in [0,1)$ and all $A\in\F_s$,
\[
\int_s^1 \PP(A \cap A_{n,t}) \, \dif t
\leq \frac{c}{\sqrt{n}} \cdot \big( \PP(A) + \EE \big[ \ind_A \cdot |\widehat{X}_{n,\usn+1/n} - \xi|^{10(1+\ell)^2} \big] \big).
\]
\item There exists $c \in (0,\infty)$ such that for all $n\in\N$, all $s \in [0,1)$ and all $A\in\F_s$,
\[
\int_s^1 \EE \big[ \ind_{A \cap A_{n,t}} \cdot |\widehat{X}_{n,\utn+1/n} - \xi|^{\alpha} \big] \, \dif t
\leq \frac{c}{n} \cdot \big( \PP(A) + \EE \big[ \ind_A \cdot |\widehat{X}_{n,\usn+1/n} - \xi|^{5 (1+\ell)^2 \alpha} \big] \big).
\]
\end{enumerate}
\end{lemma}

\begin{proof}
Let $n\in\N$, $s \in [0,1)$ and $A\in\F_s$. In the following we use $c_1,c_2, \dots \in (0,\infty)$ to denote unspecified positive constants, which neither depend on $n$ nor on $s$ nor on $A$.
  
We first prove part (i). Clearly we may assume that $s < 1-1/n$. Then $\usn \leq 1-2/n$ and we have
\[
\int_s^1 \PP(A \cap A_{n,t}) \, \dif t \leq \frac{2}{n} \cdot \PP(A) + \int_{\usn + 2/n}^{1} \PP(A \cap A_{n,t}) \, \dif t.
\]
If $t \in [\usn+2/n,1]$ then $\utn \geq \usn+2/n$, which implies $\utn-1/n \geq \usn+1/n \geq s$. We may thus apply Lemma \ref{central} to conclude that there exists $c_1 \in (0,\infty)$ such that
\begin{align*}
&\int_s^1 \PP(A \cap A_{n,t}) \, \dif t \\
&\leq \frac{c_1}{n} \cdot \PP(A) + c_1 \cdot \int_{\R} \int_{\usn+2/n}^{1} \PP \bigg( A \cap \bigg\{ |\widehat{X}_{n,\utn-(t-\utn)}-\xi| \leq \frac{c_1}{\sqrt{n}} \big(1+|z|\big) \bigg\} \bigg) \cdot e^{-\frac{z^2}{2}} \, \dif t \, \dif z.
\end{align*}
By the change-of-variable formula, we have for all $i \in \{ 1, \dots, n-1 \}$ and all $\kappa\in\R$,
\[
\int_{i/n}^{(i+1)/n} \PP\big(A \cap \big\{ |\widehat{X}_{n,\utn-(t-\utn)}-\xi| \leq \kappa \big\}\big) \, \dif t
= \int_{(i-1)/n}^{i/n} \PP\big(A \cap \big\{ |\widehat{X}_{n,t} - \xi| \leq \kappa \big\}\big) \, \dif t.
\]
Thus,
\begin{align}\label{key1}
&\int_s^1 \PP(A \cap A_{n,t}) \, \dif t \notag \\
&\leq \frac{c_1}{n} \cdot \PP(A) + c_1 \cdot \int_{\R} \int_{\usn+1/n}^{1-1/n} \PP \bigg( A \cap \bigg\{ |\widehat{X}_{n,t} - \xi| \leq \frac{c_1}{\sqrt{n}} \big(1+|z|\big) \bigg\} \bigg) \cdot e^{-\frac{z^2}{2}} \, \dif t \, \dif z.
\end{align}
By the fact that $A \in \F_{\usn+1/n}$ and by Lemma \ref{markov} we see that for all $z\in\R$,
\begin{align}\label{key2}
&\int_{\usn+1/n}^{1-1/n} \PP \bigg( A \cap \bigg\{ |\widehat{X}_{n,t} - \xi| \leq \frac{c_1}{\sqrt{n}} \big(1+|z|\big) \bigg\} \bigg) \, \dif t \notag \\
&= \EE \bigg[ \ind_A \cdot \EE \bigg[ \int_{\usn+1/n}^{1-1/n} \ind_{\{ |\widehat{X}_{n,t} - \xi| \leq \frac{c_1}{\sqrt{n}} (1+|z|) \}} \, \dif t \Big| \widehat{X}_{n,\usn+1/n} \bigg] \bigg].
\end{align}
Moreover, by Lemmas \ref{markov} and \ref{occup} we obtain that there exists $c_2 \in (0,\infty)$ such that for all $z\in\R$ and $\PP^{\widehat{X}_{n,\usn+1/n}}$-almost all $x\in\R$,
\begin{align}\label{key3}
&\EE \bigg[ \int_{\usn+1/n}^{1-1/n} \ind_{\{ |\widehat{X}_{n,t} - \xi| \leq \frac{c_1}{\sqrt{n}} (1+|z|) \}} \, \dif t \Big| \widehat{X}_{n,\usn+1/n} = x \bigg] \notag \\
&= \EE \bigg[ \int_{0}^{1-2/n-\usn} \ind_{\{ |\widehat{X}_{n,t}^x - \xi| \leq \frac{c_1}{\sqrt{n}} (1+|z|) \}} \, \dif t \bigg]
\leq c_2 \cdot \big(1+|x|^{10(1+\ell)^2}\big) \cdot \bigg( \frac{c_1}{\sqrt{n}} \big(1+|z|\big) + \frac{1}{\sqrt{n}} \bigg).
\end{align}
Combining \eqref{key2} and \eqref{key3} we conclude that for all $z\in\R$,
\begin{align}\label{key4}
&\int_{\usn+1/n}^{1-1/n} \PP \bigg( A \cap \bigg\{ |\widehat{X}_{n,t} - \xi| \leq \frac{c_1}{\sqrt{n}} \big(1+|z|\big) \bigg\} \bigg) \, \dif t \notag \\
&\leq \frac{c_2(1 + c_1)}{\sqrt{n}} \cdot \big(1+|z|\big) \cdot \EE \big[ \ind_A \cdot \big(1 + |\widehat{X}_{n,\usn+1/n}|^{10(1+\ell)^2}\big) \big] \notag \\
&\leq \frac{c_3}{\sqrt{n}} \cdot \big(1 + |\xi|^{10(1+\ell)^2}\big) \cdot \big(1+|z|\big) \cdot \big( \PP(A) + \EE \big[ \ind_A \cdot |\widehat{X}_{n,\usn+1/n} - \xi|^{10(1+\ell)^2} \big] \big).
\end{align}
Inserting \eqref{key4} into \eqref{key1} and observing that $\int_{\R} (1+|z|) \cdot e^{-z^2/2} \, \dif z < \infty$ completes the proof of part (i).
  
We next prove part (ii). Clearly,
\begin{align*}
&\int_s^1 \EE \big[ \ind_{A \cap A_{n,t}} \cdot |\widehat{X}_{n,\utn+1/n} - \xi|^{\alpha} \big] \, \dif t \\
&= \int_{s}^{\usn+1/n} \EE \big[ \ind_{A \cap A_{n,t}} \cdot |\widehat{X}_{n,\utn+1/n} - \xi|^{\alpha} \big] \, \dif t
+ \int_{\usn+1/n}^{1} \EE \big[ \ind_{A \cap A_{n,t}} \cdot |\widehat{X}_{n,\utn+1/n} - \xi|^{\alpha} \big] \, \dif t.
\end{align*}
  
If $t \in [s,\usn+1/n)$ then $\utn = \usn$ and therefore
\begin{align}\label{key5}
\int_{s}^{\usn+1/n} \EE \big[ \ind_{A \cap A_{n,t}} \cdot |\widehat{X}_{n,\utn+1/n} - \xi|^{\alpha} \big] \, \dif t
&= \int_{s}^{\usn+1/n} \EE \big[ \ind_{A \cap A_{n,t}} \cdot |\widehat{X}_{n,\usn+1/n} - \xi|^{\alpha} \big] \, \dif t \notag \\
&\leq \int_{s}^{\usn+1/n} \EE \big[ \ind_A \cdot |\widehat{X}_{n,\usn+1/n} - \xi|^{\alpha} \big] \, \dif t \notag \\
&\leq \frac{1}{n} \cdot \EE \big[ \ind_A \cdot |\widehat{X}_{n,\usn+1/n} - \xi|^{\alpha} \big].
\end{align}
  
Next, let $t \in [\usn+1/n,1]$. Clearly, we have on $A_{n,t}$,
\[
|\widehat{X}_{n,\utn+1/n} - \xi|
\leq |\widehat{X}_{n,\utn+1/n} - \widehat{X}_{n,t}| + |\widehat{X}_{n,t} - \xi|
\leq |\widehat{X}_{n,\utn+1/n} - \widehat{X}_{n,t}| + |\widehat{X}_{n,t} - \widehat{X}_{n,\utn}|.
\]
Hence, by Lemma \ref{markov},
\begin{align}\label{key6}
&\EE \big[ \ind_{A \cap A_{n,t}} \cdot |\widehat{X}_{n,\utn+1/n} - \xi|^{\alpha} \big] \notag \\
&\leq \EE \big[ \ind_A \cdot \big(|\widehat{X}_{n,\utn+1/n} - \widehat{X}_{n,t}| + |\widehat{X}_{n,t} - \widehat{X}_{n,\utn}|\big)^{\alpha} \big] \notag \\
&= \EE \big[ \ind_A \cdot \EE \big[ \big(|\widehat{X}_{n,\utn+1/n} - \widehat{X}_{n,t}| + |\widehat{X}_{n,t} - \widehat{X}_{n,\utn}|\big)^{\alpha} | \widehat{X}_{n,\usn+1/n} \big] \big].
\end{align}
  
If $t \geq \usn+1/n$ then $\utn \geq \usn+1/n$. Hence, by Lemma \ref{tamprop2} and \ref{markov} we obtain that there exist $c_1,c_2 \in (0,\infty)$ such that for all $t \in [\usn+1/n,1]$ and $\PP^{\widehat{X}_{n,\usn+1/n}}$-almost all $x\in\R$,
\begin{align}\label{key7}
&\EE \big[ \big(|\widehat{X}_{n,\utn+1/n} - \widehat{X}_{n,t}| + |\widehat{X}_{n,t} - \widehat{X}_{n,\utn}|\big)^{\alpha} | \widehat{X}_{n,\usn+1/n} = x \big] \notag \\
&= \EE \big[ \big(|\widehat{X}_{n,\utn-\usn}^x - \widehat{X}_{n,t-\usn-1/n}^x| + |\widehat{X}_{n,t-\usn-1/n}^x - \widehat{X}_{n,\utn-\usn-1/n}^x|\big)^{\alpha} \big] \notag \\
&\leq c_1 \cdot \big(1+|x|^{5 (1+\ell)^2 \alpha}\big) \cdot 1/n
\leq c_2 \cdot \big(1 + |x-\xi|^{5 (1+\ell)^2 \alpha}\big) \cdot 1/n.
\end{align}
It follows from \eqref{key6} and \eqref{key7} that
\begin{align}\label{key8}
&\int_{\usn+1/n}^{1} \EE \big[ \ind_{A \cap A_{n,t}} \cdot |\widehat{X}_{n,\utn+1/n} - \xi|^{\alpha} \big] \, \dif t \notag \\
&\leq \frac{c_2}{n} \cdot \int_{\usn+1/n}^{1} \EE \big[ \ind_A \cdot \big(1 + |\widehat{X}_{n,\usn+1/n} - \xi|^{5 (1+\ell)^2 \alpha}\big) \big] \, \dif t \notag \\
&\leq \frac{c_2}{n} \cdot \big( \PP(A) + \EE \big[ \ind_A \cdot |\widehat{X}_{n,\usn+1/n} - \xi|^{5 (1+\ell)^2 \alpha} \big] \big).
\end{align}
Combining \eqref{key5} with \eqref{key8} completes the proof of part (ii).
\end{proof}

We are ready to establish the main result of this subsection, which provides a $p$-th mean estimate of the Lebesgue measure of the set of times $t$ of a sign change of $\widehat{X}_{n,t}-\xi$ relative to the sign of $\widehat{X}_{n,\utn}-\xi$.

\begin{prop} \label{prop}
Let Assumption \ref{assum2} hold. Let $\xi\in\R$ satisfy $\sigma(\xi) \neq 0$ and $p \in [1,\infty)$. Then there exists $c \in (0,\infty)$ such that for all $n\in\N$,
\[
\bigg(\EE\bigg[ \bigg|\int_0^1 \ind_{\{ (\widehat{X}_{n,t}-\xi) \cdot (\widehat{X}_{n,\utn}-\xi) \leq 0 \}} \, \dif t\bigg|^p \bigg]\bigg)^{1/p} \leq \frac{c}{\sqrt{n}}.
\]
\end{prop}

\begin{proof}
Clearly, it suffices to consider only the case $p\in\N$. For $n\in\N$ and $t \in [0,1]$ put $A_{n,t} = \{ (\widehat{X}_{n,t}-\xi) \cdot (\widehat{X}_{n,\utn}-\xi) \leq 0 \}$ as in Lemma \ref{key}, and for $n,p\in\N$ let
\[
a_{n,p} = \EE\bigg[ \bigg(\int_0^1 \ind_{A_{n,t}} \, \dif t\bigg)^p \bigg].
\]
We prove by induction on $p$ that for every $p\in\N$ there exists $c \in (0,\infty)$ such that for all $n\in\N$,
\begin{equation}\label{prop1}
a_{n,p} \leq c \cdot n^{-p/2}.
\end{equation}
  
First assume that $p=1$. Using Lemma \ref{key}(i) with $s=0$ and $A=\Omega$ we obtain that there exist $c_1,c_2 \in (0,\infty)$ such that for all $n\in\N$,
\begin{align*}
a_{n,1} &= \int_0^1 \PP(A_{n,t}) \, \dif t \\
&\leq \frac{c_1}{\sqrt{n}} \cdot \big(1 + \EE\big[ |\widehat{X}_{n,1/n}-\xi|^{10(1+\ell)^2} \big]\big) \\
&\leq \frac{c_2}{\sqrt{n}} \cdot \Big(1 + |\xi|^{10(1+\ell)^2} + \sup_{j\in\N} \EE\big[ \|\widehat{X}_j\|_{\infty}^{10(1+\ell)^2} \big]\Big).
\end{align*}
Observing Lemma \ref{tamprop1} we thus see that \eqref{prop1} holds for $p=1$.

Next, let $q\in\N$ and assume that \eqref{prop1} holds for all $p \in \{ 1,\dots,q \}$. Clearly, for all $n\in\N$,
\[
a_{n,q+1} = (q+1)! \cdot \int_0^1 \int_{t_1}^1 \cdots \int_{t_q}^1 \PP(A_{n,t_1} \cap A_{n,t_2} \cap \cdots \cap A_{n,t_{q+1}}) \, \dif t_{q+1} \dots \, \dif t_2 \, \dif t_1.
\]

First applying Lemma \ref{key}(i) with $A = A_{n,t_1} \cap \cdots \cap A_{n,t_q}$ and $s = t_q$, then applying $(q-1)$-times Lemma \ref{key}(ii) with $A = A_{n,t_1} \cap \cdots \cap A_{n,t_j}$ and $s = t_j$ for $j = q-1,\dots,1$, and finally applying Lemma \ref{key}(ii) with $A=\Omega$ and $s=0$ we conclude that there exist constants $c_3,c_4,c_5 \in (0,\infty)$ such that for all $n\in\N$,
\begin{align*}
a_{n,q+1} &\leq \frac{c_3}{\sqrt{n}} \cdot \bigg(a_{n,q} + \int_0^1 \cdots \int_{t_{q-1}}^1 \EE\big[ \ind_{A_{n,t_1} \cap \cdots \cap A_{n,t_q}} \cdot |\widehat{X}_{n,\underline{t_q}_n+1/n}-\xi|^{10(1+\ell)^2} \big] \, \dif t_q \dots \, \dif t_1\bigg) \\
&\leq c_4 \cdot \bigg(\frac{a_{n,q}}{\sqrt{n}} + \frac{a_{n,q-1}}{n^{3/2}} + \cdots + \frac{a_{n,1}}{n^{q-1/2}} \\
&\qquad \qquad + \frac{1}{n^{q-1/2}} \cdot \int_0^1 \EE\big[ \ind_{A_{n,t_1}} \cdot |\widehat{X}_{n,\underline{t_1}_n+1/n}-\xi|^{2 \cdot 5^q(1+\ell)^{2q}} \big] \, \dif t_1\bigg) \\
&\leq c_4 \cdot \bigg(\frac{a_{n,q}}{\sqrt{n}} + \frac{a_{n,q-1}}{n^{3/2}} + \cdots + \frac{a_{n,1}}{n^{q-1/2}} \\
&\qquad \qquad + \frac{c_5}{n^{q+1/2}} \cdot \Big(1 + |\xi|^{2 \cdot 5^q(1+\ell)^{2q}} + \sup_{j\in\N} \EE\big[ \|\widehat{X}_j\|_{\infty}^{2 \cdot 5^q(1+\ell)^{2q}} \big]\Big)\bigg).
\end{align*}
Employing Lemma \ref{tamprop1} and the induction hypothesis yields the validity of \eqref{prop1} for $p = q+1$, which finishes the proof of the proposition.
\end{proof}

\subsection{The transformed equation}

In this subsection, we provide some estimates for the time-continuous tamed Euler scheme associated to the SDE \eqref{sde2}.

For every $n\in\N$, let $\widehat{Z}_n = (\widehat{X}_{n,t})_{t \in [0,1]}$ denote the time-continuous tamed Euler scheme with step-size $1/n$ associated to the SDE \eqref{sde2}, i.e. $\widehat{Z}_{n,0} = G(x_0)$ and
\[
\widehat{Z}_{n,t} = \widehat{Z}_{n,i/n} + \frac{\widetilde{\mu}(\widehat{Z}_{n,i/n})}{1 + 1/n \cdot |\widetilde{\mu}(\widehat{Z}_{n,i/n})|} \cdot (t-i/n) + \widetilde{\sigma}(\widehat{Z}_{n,i/n}) \cdot (W_t-W_{i/n})
\]
for $t \in (i/n,(i+1)/n]$ and $i \in \{ 0,\dots,n-1 \}$.

\begin{lemma}
Let Assumption \ref{assum2} hold. Then $\widetilde{\mu}$ is locally Lipschitz continuous, $\widetilde{\sigma}$ is Lipschitz continuous, and there exists $c \in (0,\infty)$ such that for all $x,y\in\R$,
\begin{align*}
&(x-y) \cdot (\widetilde{\mu}(x)-\widetilde{\mu}(y)) 
\leq c \cdot |x-y|^2, \\
&|\widetilde{\mu}(x)-\widetilde{\mu}(y)| 
\leq c \cdot \big(1+|x|^{\ell}+|y|^{\ell}\big) \cdot |x-y|.
\end{align*}
\end{lemma}

\begin{proof}
It is straightforward to check that $\widetilde{\mu}$ is locally Lipschitz continuous and $\widetilde{\sigma}$ is Lipschitz continuous.

By Lemma \ref{transform} we obtain that there exists $n_0\in\N$ such that for all $x\in\R$ with $|x| > n_0$,
\[
\widetilde{\mu}(x) = \mu(x).
\]
Hence, by Lemma \ref{assum2} we obtain that there exists $c_1 \in (0,\infty)$ such
that for all $x,y \in (-\infty,-n_0\wedge\xi_1)$ and all $x,y \in (n_0\vee\xi_k,\infty)$,
\begin{align}\label{trcoe2}
&(x-y) \cdot (\widetilde{\mu}(x)-\widetilde{\mu}(y)) 
\leq c_1 \cdot |x-y|^2, \notag \\
&|\widetilde{\mu}(x)-\widetilde{\mu}(y)| 
\leq c_1 \cdot \big(1+|x|^{\ell}+|y|^{\ell}\big) \cdot |x-y|.
\end{align}

Using the local Lipschitz continuity of $\widetilde{\mu}$ we obtain that there exists $c_2 \in (0,\infty)$ such that for all $x,y \in [-n_0\wedge\xi_1,n_0\vee\xi_k]$,
\[
|\widetilde{\mu}(x)-\widetilde{\mu}(y)| \leq c_2 \cdot |x-y|,
\]
which, together with \eqref{trcoe2}, completes the proof.
\end{proof}

\begin{lemma}\label{trtamprop}
Let Assumption \ref{assum2} hold and $p \in [1,\infty)$. Then there exists $c \in (0,\infty)$ such that for all $n\in\N$, all $\delta \in [0,1]$ and all $t \in [0,1-\delta]$,
\[
\Big(\EE\Big[ \sup_{s \in [t,t+\delta]} |\widehat{Z}_{n,s}-\widehat{Z}_{n,t}|^p \Big]\Big)^{1/p} \leq c \cdot \sqrt{\delta}.
\]
In particular,
\[
\sup_{n\in\N} \EE\big[ \|\widehat{Z}_n\|_{\infty}^p \big] < \infty.
\]
\end{lemma}

\begin{proof}
See the proof of Lemma \ref{tamprop1}.
\end{proof}

\begin{lemma}\label{trtamprop2}
Let Assumption \ref{assum2} hold and $p \in [1,\infty)$. Then there exists $c \in (0,\infty)$ such that for all $n\in\N$,
\[
\big(\EE\big[ \|Z-\widehat{Z}_n\|_{\infty}^p \big]\big)^{1/p} \leq \frac{c}{\sqrt{n}}.
\]
\end{lemma}

\begin{proof}
See the proof of \cite[Theorem 1.1]{MH}.
\end{proof}

Finally, we provide an estimate for the transformed time-continuous tamed Euler scheme $G\circ\widehat{X}_n = (G(\widehat{X}_{n,t}))_{t \in [0,1]}$.

\begin{lemma}\label{trtamprop3}
Let Assumption \ref{assum2} hold and $p \in [1,\infty)$. Then there exists $c \in (0,\infty)$ such that for all $n\in\N$,
\[
\EE\big[ \|G\circ\widehat{X}_n\|_{\infty}^p \big] \leq c.
\]
\end{lemma}

\begin{proof}
See the proof of \cite[Lemma 10]{TM}.
\end{proof}

\subsection{\texorpdfstring{$L_{p}$-error estimate}{Lp-error estimate}}

We are now ready to establish the main result of this article.

\begin{theorem}\label{thm}
Let Assumption \ref{assum2} hold and $p \in [1,\infty)$. Then there exists $c \in (0,\infty)$ such that for all $n\in\N$,
\[
\big(\EE\big[ \|X-\widehat{X}_n\|_{\infty}^p \big]\big)^{1/p} \leq \frac{c}{\sqrt{n}}.
\]
\end{theorem}

\begin{proof}
Without loss of generality we may assume that $p \in [4,\infty)$. For every $n\in\N$ we define a function $u_n : [0,1] \to [0,\infty)$ by
\[
u_n(t) = \EE\Big[ \sup_{s \in [0,t]} |G(\widehat{X}_{n,s})-\widehat{Z}_{n,s}|^p \Big].
\]
Note that the functions $u_n, n\in\N$, are well-defined and bounded due to Lemma \ref{trtamprop} and Lemma \ref{trtamprop3}.

Below we show that there exists $c \in (0,\infty)$ such that for all $n\in\N$ and all $t \in [0,1]$,
\begin{equation}\label{thm1}
u_n(t) \leq c \cdot \bigg(\frac{1}{n^{p/2}} + \sum_{i=1}^k \EE\bigg[\bigg|\int_0^t \ind_{\{ (\widehat{X}_{n,s}-\xi_i) \cdot (\widehat{X}_{n,\usn}-\xi_i) \leq 0 \}} \, \dif s\bigg|^p\bigg] + \int_0^t u_n(s) \, \dif s\bigg).
\end{equation}
Using Lemma \ref{prop} we conclude from \eqref{thm1} that there exists $c \in (0,\infty)$ such that for all $n\in\N$ and all $t \in [0,1]$,
\[
u_n(t) \leq c \cdot \bigg(\frac{1}{n^{p/2}} + \int_0^t u_n(s) \, \dif s\bigg).
\]
By the Gronwall inequality it then follows that there exists $c \in (0,\infty)$ such that for all $n\in\N$,
\begin{equation}\label{thm2}
u_n(1) \leq \frac{c}{n^{p/2}}.
\end{equation}
Using the fact that $G^{-1}$ is Lipschitz continuous, see Lemma \ref{transform}(i), as well as Lemma \ref{trtamprop2} and \eqref{thm2} we conclude that there exists $c_1,c_2 \in (0,\infty)$ such that for all $n\in\N$,
\[
\EE\big[ \|X-\widehat{X}_n\|_{\infty}^p \big] \leq c_1 \cdot \EE\big[ \|Z-G\circ\widehat{X}_n\|_{\infty}^p \big] \leq 2^p \cdot c_1 \cdot \big(\EE\big[ \|Z-\widehat{Z}_n\|_{\infty}^p \big] + u_n(1)\big) \leq \frac{c_2}{n^{p/2}},
\]
which yields the statement of Theorem \ref{thm}.

It remains to prove \eqref{thm1}. Let $n\in\N$. Clearly, for every $t \in [0,1]$,
\[
\widehat{Z}_{n,t} = G(x_0) + \int_0^t \frac{\widetilde{\mu}(\widehat{Z}_{n,\usn})}{1 + 1/n \cdot |\widetilde{\mu}(\widehat{Z}_{n,\usn})|} \, \dif s + \int_0^t \widetilde{\sigma}(\widehat{Z}_{n,\usn}) \, \dif W_s.
\]
Since $G^{\prime}$ is absolutely continuous, we may apply the It\^{o} formula, see e.g. \cite[Problem 3.7.3]{IK}, to obtain that $\PP$-a.s. for all $t \in [0,1]$,
\begin{align*}
G(\widehat{X}_{n,t}) 
&= G(x_0) + \int_0^t \bigg( G^{\prime}(\widehat{X}_{n,s}) \cdot \frac{\mu(\widehat{X}_{n,\usn})}{1 + 1/n \cdot |\mu(\widehat{X}_{n,\usn})|} + \frac12 G^{\dprime}(\widehat{X}_{n,s}) \cdot \sigma^2(\widehat{X}_{n,\usn}) \bigg) \, \dif s \\
&\quad + \int_0^t G^{\prime}(\widehat{X}_{n,s}) \cdot \sigma(\widehat{X}_{n,\usn}) \, \dif W_s.
\end{align*}
The It\^{o} formula gives that $\PP$-a.s. for all $t \in [0,1]$,
\[
|G(\widehat{X}_{n,t})-\widehat{Z}_{n,t}|^2 = \sum_{i=1}^8 V_{n,i,t},
\]
where
\begin{align*}
V_{n,1,t} &= \int_0^t |G^{\prime}(\widehat{X}_{n,s}) \cdot \sigma(\widehat{X}_{n,\usn}) - \widetilde{\sigma}(\widehat{Z}_{n,\usn}))|^2 \, \dif s, \\
V_{n,2,t} &= \int_0^t 2\big(G(\widehat{X}_{n,s})-\widehat{Z}_{n,s}\big) \cdot \big(G^{\prime}(\widehat{X}_{n,s}) \cdot \sigma(\widehat{X}_{n,\usn}) - \widetilde{\sigma}(\widehat{Z}_{n,\usn})\big) \, \dif W_s, \\
V_{n,3,t} &= \int_0^t 2\big(G(\widehat{X}_{n,s})-\widehat{Z}_{n,s}\big) \cdot \big(G^{\prime}(\widehat{X}_{n,s}) - G^{\prime}(\widehat{X}_{n,\usn})\big) \cdot \frac{\mu(\widehat{X}_{n,\usn})}{1 + 1/n \cdot |\mu(\widehat{X}_{n,\usn})|} \, \dif s, \\
V_{n,4,t} &= \int_0^t 2\big(G(\widehat{X}_{n,s})-\widehat{Z}_{n,s}\big) \cdot \big(\big(\widetilde{\mu}(G(\widehat{X}_{n,\usn})) - \widetilde{\mu}(\widehat{Z}_{n,\usn})\big) - \big(\widetilde{\mu}(G(\widehat{X}_{n,s})) - \widetilde{\mu}(\widehat{Z}_{n,s})\big)\big) \, \dif s, \\
V_{n,5,t} &= - \int_0^t 2\big(G(\widehat{X}_{n,s})-\widehat{Z}_{n,s}\big) \cdot G^{\prime}(\widehat{X}_{n,\usn}) \cdot \mu(\widehat{X}_{n,\usn}) \cdot \frac{1/n \cdot |\mu(\widehat{X}_{n,\usn})|}{1 + 1/n \cdot |\mu(\widehat{X}_{n,\usn})|} \, \dif s, \\
V_{n,6,t} &= \int_0^t 2\big(G(\widehat{X}_{n,s})-\widehat{Z}_{n,s}\big) \cdot \widetilde{\mu}(\widehat{Z}_{n,\usn}) \cdot \frac{1/n \cdot |\widetilde{\mu}(\widehat{Z}_{n,\usn})|}{1 + 1/n \cdot |\widetilde{\mu}(\widehat{Z}_{n,\usn})|} \, \dif s, \\
V_{n,7,t} &= \int_0^t 2\big(G(\widehat{X}_{n,s})-\widehat{Z}_{n,s}\big) \cdot \big(\widetilde{\mu}(G(\widehat{X}_{n,s})) - \widetilde{\mu}(\widehat{Z}_{n,s})\big) \, \dif s, \\
V_{n,8,t} &= \int_0^t \big(G(\widehat{X}_{n,s})-\widehat{Z}_{n,s}\big) \cdot \big(G^{\dprime}(\widehat{X}_{n,s}) \cdot \sigma^2(\widehat{X}_{n,\usn}) - G^{\dprime}(\widehat{X}_{n,\usn}) \cdot \sigma^2(\widehat{X}_{n,\usn})\big) \, \dif s.
\end{align*}

Using the H\"{o}lder inequality, the Burkholder--Davis--Gundy inequality, Lemma \ref{tamprop1}, \ref{trtamprop} and \eqref{ine2} we conclude that there exists $c \in (0,\infty)$ such that for all $n\in\N$ and all $t \in [0,1]$,
\begin{align}\label{thm3}
&\EE\Big[ \sup_{s \in [0,t]} |V_{n,i,s}|^{p/2} \Big] \leq \frac{c}{n^{p/2}} + c \cdot \int_0^t u_n(s) \, \dif s, \quad 1 \leq i \leq 4, \notag \\
&\EE\Big[ \sup_{s \in [0,t]} |V_{n,i,s}|^{p/2} \Big] \leq \frac{c}{n^p} + c \cdot \int_0^t u_n(s) \, \dif s, \quad 5 \leq i \leq 6.
\end{align}

Using the one-sided Lipschitz continuity of $\widetilde{\mu}$, we see that there exists $c_1 \in (0,\infty)$ such that for all $n\in\N$ and all $t \in [0,1]$,
\[
V_{n,7,t} \leq \widetilde{V}_{n,7,t}
= c_1 \cdot \int_0^t |G(\widehat{X}_{n,s})-\widehat{Z}_{n,s}|^2 \, \dif s.
\]
Using the H\"{o}lder inequality we conclude that there exists $c_2 \in (0,\infty)$ such that for all $n\in\N$ and all $t \in [0,1]$,
\begin{equation}\label{thm4}
\EE\Big[ \sup_{s \in [0,t]} |\widetilde{V}_{n,7,s}|^{p/2} \Big] \leq c_2 \cdot \int_0^t u_n(s) \, \dif s.
\end{equation}

For estimating $\EE[ \sup_{s \in [0,t]} |V_{n,8,s}|^{p/2} ]$ we put
\[
B = \bigg( \bigcup_{i=1}^{k+1} (\xi_{i-1},\xi_i)^2 \bigg)^c
\]
and we note that 
\[
B = \bigcup_{i=1}^{k} \{ (x,y) \in \R^2 : (x-\xi_i) \cdot (y-\xi_i) \leq 0 \}.
\]
Using Lemma \ref{transform}(ii) and the Lipschitz continuity of $\sigma$, we see that there exists $c \in (0,\infty)$ such that for all $x,y\in\R$,
\[
|G^{\dprime}(x) \cdot \sigma^2(y) - G^{\dprime}(y) \cdot \sigma^2(y)| \leq 
\begin{cases}
c \cdot (1 + y^2) \cdot |x-y|, & (x,y) \in B^c, \\
c \cdot (1 + y^2), & (x,y) \in B.
\end{cases}
\]
Hence, using \eqref{ine} we obtain that there exists $c \in (0,\infty)$ such that for all $n\in\N$ and all $t \in [0,1]$,
\begin{align}\label{thm5}
&\bigg|\int_0^t |G^{\dprime}(\widehat{X}_{n,s}) \cdot \sigma^2(\widehat{X}_{n,\usn}) - G^{\dprime}(\widehat{X}_{n,\usn}) \cdot \sigma^2(\widehat{X}_{n,\usn})| \, \dif s\bigg|^p \notag \\
&\leq c \cdot \bigg(\bigg|\int_0^t \big(1 + \widehat{X}_{n,\usn}^2\big) \cdot |\widehat{X}_{n,s}-\widehat{X}_{n,\usn}| \, \dif s\bigg|^p + \bigg|\int_0^t \big(1 + \widehat{X}_{n,\usn}^2\big) \cdot \ind_{\{ (\widehat{X}_{n,s},\widehat{X}_{n,\usn}) \in B \}} \, \dif s\bigg|^p\bigg).
\end{align}
Using the H\"{o}lder inequality and Lemma \ref{tamprop1} we obtain that there exists $c \in (0,\infty)$ such that for all $n\in\N$ and all $t \in [0,1]$,
\begin{equation}\label{thm6}
\EE\bigg[\bigg|\int_0^t \big(1 + \widehat{X}_{n,\usn}^2\big) \cdot |\widehat{X}_{n,s}-\widehat{X}_{n,\usn}| \, \dif s\bigg|^p\bigg] \leq \frac{c}{n^{p/2}}.
\end{equation}
Furthermore, for all $i \in \{ 1,\dots,k \}$, all $n\in\N$ and all $s \in [0,1]$,
\begin{align*}
|\widehat{X}_{n,\usn}| \cdot \ind_{\{ (\widehat{X}_{n,s}-\xi_i) \cdot (\widehat{X}_{n,\usn}-\xi_i) \leq 0 \}} 
&\leq \big(|\xi_i| + |\widehat{X}_{n,\usn}-\xi_i|\big) \cdot \ind_{\{ (\widehat{X}_{n,s}-\xi_i) \cdot (\widehat{X}_{n,\usn}-\xi_i) \leq 0 \}} \\
&\leq \big(|\xi_i| + |\widehat{X}_{n,\usn}-\widehat{X}_{n,s}|\big) \cdot \ind_{\{ (\widehat{X}_{n,s}-\xi_i) \cdot (\widehat{X}_{n,\usn}-\xi_i) \leq 0 \}},
\end{align*}
which yields that for all $n\in\N$ and all $s \in [0,1]$,
\[
\big(1 + \widehat{X}_{n,\usn}^2\big) \cdot \ind_{\{ (\widehat{X}_{n,s},\widehat{X}_{n,\usn}) \in B \}} \leq \Big(1 + 2\max_{1 \leq i \leq k} \xi_i^2\Big) \cdot \sum_{i=1}^k \ind_{\{ (\widehat{X}_{n,s}-\xi_i) \cdot (\widehat{X}_{n,\usn}-\xi_i) \leq 0 \}} + 2\big(\widehat{X}_{n,\usn}-\widehat{X}_{n,s}\big)^2.
\]
By the latter inequality and Lemma \ref{tamprop1} we conclude that there exists $c \in (0,\infty)$ such that for all $n\in\N$ and all $t \in [0,1]$,
\begin{align}\label{thm7}
&\EE\bigg[\bigg|\int_0^t \big(1 + \widehat{X}_{n,\usn}^2\big) \cdot \ind_{\{ (\widehat{X}_{n,s},\widehat{X}_{n,\usn}) \in B \}} \, \dif s\bigg|^p\bigg] \notag \\
&\leq c \cdot \sum_{i=1}^k \EE\bigg[\bigg|\int_0^t \ind_{\{ (\widehat{X}_{n,s}-\xi_i) \cdot (\widehat{X}_{n,\usn}-\xi_i) \leq 0 \}} \, \dif s\bigg|^p\bigg] + \frac{c}{n^p}.
\end{align}
Combining \eqref{thm5}, \eqref{thm6} and \eqref{thm7} we obtain that there exists $c \in (0,\infty)$ such that for all $n\in\N$ and all $t \in [0,1]$,
\begin{align*}
&\EE\bigg[\bigg|\int_0^t |G^{\dprime}(\widehat{X}_{n,s}) \cdot \sigma^2(\widehat{X}_{n,\usn}) - G^{\dprime}(\widehat{X}_{n,\usn}) \cdot \sigma^2(\widehat{X}_{n,\usn})| \, \dif s\bigg|^p\bigg] \\
&\leq \frac{c}{n^{p/2}} + c \cdot \sum_{i=1}^k \EE\bigg[\bigg|\int_0^t \ind_{\{ (\widehat{X}_{n,s}-\xi_i) \cdot (\widehat{X}_{n,\usn}-\xi_i) \leq 0 \}} \, \dif s\bigg|^p\bigg].
\end{align*}
Hence, using the H\"{o}lder inequality and the fact that for all $a,b\in\R$ and all $\eps \in (0,\infty)$,
\[
ab \leq \frac{\eps}{2} \cdot a^2 + \frac{1}{2\eps} \cdot b^2,
\]
we obtain that there exists $c \in (0,\infty)$ such that for all $n\in\N$, all $t \in [0,1]$ and all $\eps \in (0,\infty)$,
\begin{align*}
&\EE\Big[ \sup_{s \in [0,t]} |V_{n,8,s}|^{p/2} \Big] \\
&\leq \EE\bigg[ \sup_{s \in [0,t]} |G(\widehat{X}_{n,s})-\widehat{Z}_{n,s}|^{p/2} \cdot \bigg|\int_0^t |G^{\dprime}(\widehat{X}_{n,s}) \cdot \sigma^2(\widehat{X}_{n,\usn}) - G^{\dprime}(\widehat{X}_{n,\usn}) \cdot \sigma^2(\widehat{X}_{n,\usn})| \, \dif s\bigg|^{p/2} \bigg] \\
&\leq \sqrt{u_n(t)} \cdot \bigg(\EE\bigg[ \bigg|\int_0^t |G^{\dprime}(\widehat{X}_{n,s}) \cdot \sigma^2(\widehat{X}_{n,\usn}) - G^{\dprime}(\widehat{X}_{n,\usn}) \cdot \sigma^2(\widehat{X}_{n,\usn})| \, \dif s\bigg|^p \bigg]\bigg)^{1/2} \\
&\leq \frac{\eps}{2} \cdot u_n(t) + \frac{c}{2\eps} \cdot \bigg(\frac{1}{n^{p/2}} + \sum_{i=1}^k \EE\bigg[\bigg|\int_0^t \ind_{\{ (\widehat{X}_{n,s}-\xi_i) \cdot (\widehat{X}_{n,\usn}-\xi_i) \leq 0 \}} \, \dif s\bigg|^p\bigg]\bigg),
\end{align*}
which, together with \eqref{thm3} and \eqref{thm4}, yields that there exists $c \in (0,\infty)$ such that for all $n\in\N$, all $t \in [0,1]$ and all $\eps \in (0,\infty)$,
\[
u_n(t) \leq c \cdot \bigg(\eps \cdot u_n(t) + \frac{1+1/\eps}{n^{p/2}} + \sum_{i=1}^k \EE\bigg[\bigg|\int_0^t \ind_{\{ (\widehat{X}_{n,s}-\xi_i) \cdot (\widehat{X}_{n,\usn}-\xi_i) \leq 0 \}} \, \dif s\bigg|^p\bigg] + \int_0^t u_n(s) \, \dif s\bigg).
\]
Let $\eps = 1/(2c)$ then we obtain that for all $n\in\N$ and all $t \in [0,1]$,
\[
u_n(t) \leq 2c \cdot \bigg(\frac{1+2c}{n^{p/2}} + \sum_{i=1}^k \EE\bigg[\bigg|\int_0^t \ind_{\{ (\widehat{X}_{n,s}-\xi_i) \cdot (\widehat{X}_{n,\usn}-\xi_i) \leq 0 \}} \, \dif s\bigg|^p\bigg] + \int_0^t u_n(s) \, \dif s\bigg).
\]
This finishes the proof of \eqref{thm1}.
\end{proof}

\section{Numerical experiment}

In this section, we provide a numerical example, testing the $L_2$-error rate of the the tamed Euler scheme.

We choose $x_0 = 1$, 
\[
\mu(x) = \ind_{(1,\infty)}(x) - x^5 \quad
\textrm{and} \quad
\sigma(x) = x
\]
for all $x\in\R$. The SDE \eqref{sde1} thus reads as
\begin{equation}\label{sde4}
\begin{aligned}
\dif X_t &= \big(\ind_{(1,\infty)}(X_t) - (X_t)^5\big) \, \dif t + X_t \, \dif W_t, \quad t \geq 0, \\
X_0 &= 1.
\end{aligned}
\end{equation}

\begin{figure}[ht]
\centering
\includegraphics[width=0.6\linewidth]{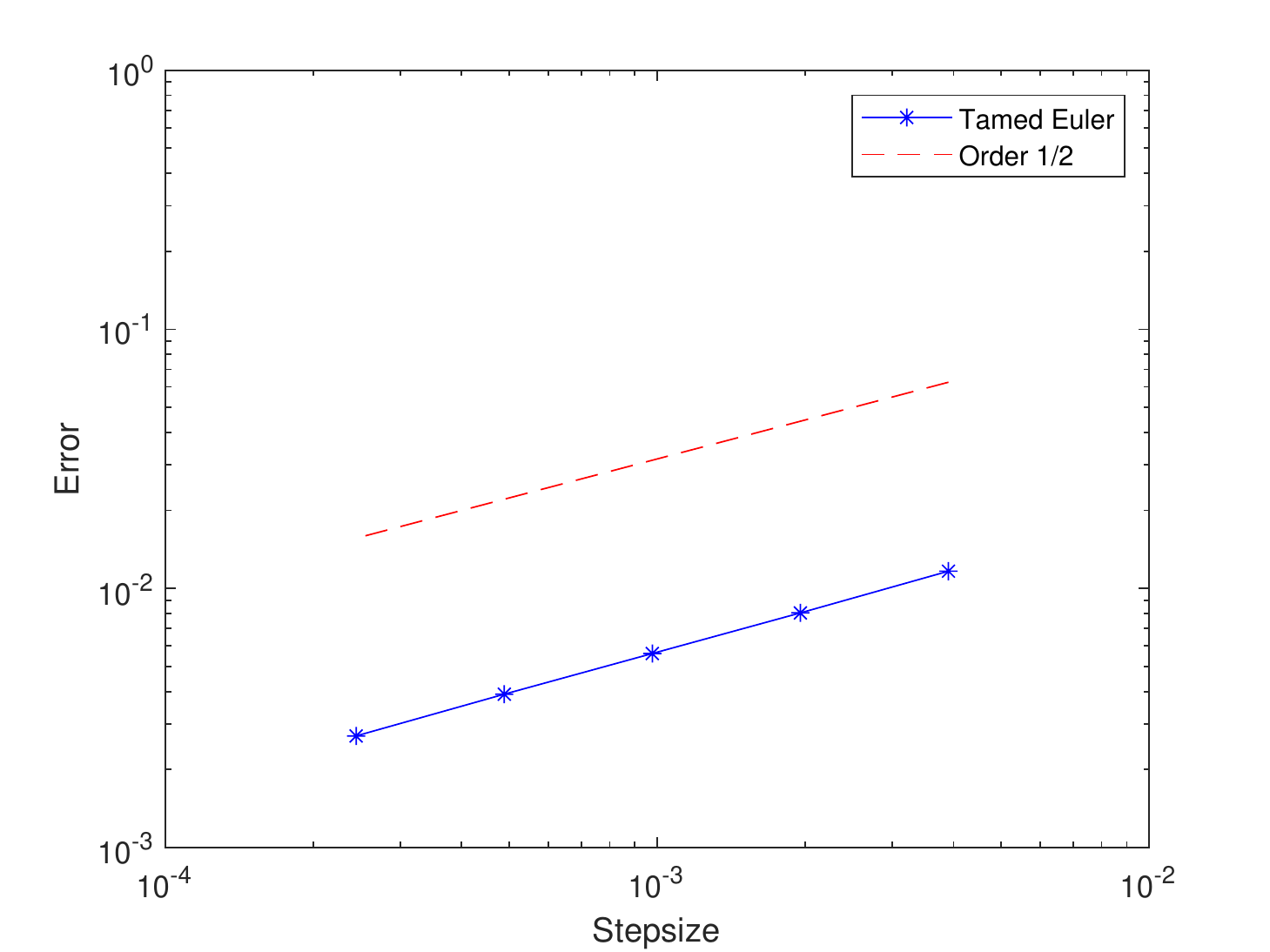}
\caption{Root mean square approximation error versus time-stepsize.}
\label{fig}
\end{figure}

To clearly display the convergence rate of the tamed Euler scheme at the time $T = 1$, we plot our approximation of the root mean-square errors as a function of the stepsize $h = 1/n$ in log-log scale in Figure \ref{fig}, where the expectation is approximated by the mean of $10^5$ independent realizations. Since there is no explicit solution available for the SDE \eqref{sde4}, we use the tamed Euler scheme with a fine stepsize $h = 2^{-16}$ to obtain the reference solution. As predicted, the tamed Euler scheme gives errors that decrease proportional to $h$, indicating that the scheme converges strongly with the standard order $1/2$ to the exact solution of the SDE.

\bibliographystyle{acm}
\bibliography{bibfile}

\end{document}